\documentclass[a4paper,12pt]{amsart}
\usepackage[margin=2.5cm]{geometry}
\usepackage{amsmath,amsthm,amssymb,amsfonts,mathrsfs,mathtools,bm,tikz,amscd}
\usepackage[T1]{fontenc}
\usepackage{enumerate}
\usepackage{lmodern}
\usepackage[utf8]{inputenc}
\usepackage[numbers]{natbib}
\usepackage{hyperref}
\usepackage{url}
\usepackage{xcolor}
\allowdisplaybreaks
\numberwithin{equation}{section}
\usetikzlibrary{arrows.meta}

\theoremstyle{definition}
\newtheorem{dfn}{Definition}[section]

\newtheorem{rmk}[dfn]{Remark}

\theoremstyle{plain} 
\newtheorem{thm}[dfn]{Theorem}
\newtheorem{prop}[dfn]{Proposition}

\newtheorem{lem}[dfn]{Lemma}

\newcommand{\N}{\mathbb{N}}
\newcommand{\Z}{\mathbb{Z}}
\newcommand{\R}{\mathbb{R}}

\usepackage{enumitem}

\title[Existence of Kelvin-Invariant Positive Solutions]{Existence of Kelvin-Invariant Positive Solutions
for Critical Elliptic Equations with Variable Coefficients 
via Profile Decomposition}

\author{Yohei Sato}
\address{Department of Mathematics, Saitama University, Shimo-Okubo 255, Sakura-ku Saitama-shi, 338-8570, JAPAN}
\email{ ysato@rimath.saitama-u.ac.jp}

\author{Reo Suzuki}
\address{Department of Mathematics, Saitama University, Shimo-Okubo 255, Sakura-ku Saitama-shi, 338-8570, JAPAN}
\email{ suzuki.r.990@ms.saitama-u.ac.jp}

\begin{document}

\begin{abstract}
In this paper, we consider the following critical nonlinear elliptic equation:
\begin{equation}
\tag{$\mathcal{P}_a$}
\left\{ 
\begin{aligned}    
- & \Delta u = a(x) |u|^{2^*-2}u \quad \text{in } \mathbb{R}^N\\
& u \in \mathcal{D}^{1,2}(\mathbb{R}^N)
\end{aligned} 
\right.
\end{equation}
where $N \ge 3$, $2^* = \frac{2N}{N - 2}$, 
$a(x) \in C(\mathbb{R}^N,\mathbb{R})$ is 
a positive function that is invariant under the map $x \to -\frac{x}{|x|^2}$. 
Under some assumptions on $a(x)$, we show 
the existence of a positive solution to 
($\mathcal{P}_a$) that is invariant under 
the Kelvin transform. 
The symmetry condition imposed here 
is substantially weaker than the invariance 
under a noncompact symmetry group that 
is typically assumed in the literature.
The key to the proof is a classification of the 
Palais--Smale sequences of 
the associated energy functional.
To this end, we establish a new abstract 
profile decomposition theorem incorporating 
symmetries such as the Kelvin transform.

\medskip



\end{abstract}

\maketitle

\section{Introduction}
In this paper, we consider the following critical nonlinear elliptic equation:
\begin{equation}
\label{Pa}
\tag{$\mathcal{P}_a$}
\left\{ 
\begin{aligned}    
	- & \Delta u = a(x) |u|^{2^*-2}u \quad \text{in } \mathbb{R}^N\\
	& u \in \mathcal{D}^{1,2}(\mathbb{R}^N)
\end{aligned} 
\right.
\end{equation}
where $N \ge 3$, $2^* = \frac{2N}{N - 2}$ is the critical Sobolev exponent, $\mathcal{D}^{1,2}(\mathbb{R}^N) \coloneqq \overline{C_0^\infty(\mathbb{R}^N)}^{\|\nabla \cdot\|_{L^2(\mathbb{R}^N)}}$, and $a(x) \in C(\mathbb{R}^N,\mathbb{R})$ satisfies the following conditions.

\begin{itemize}

\item[(a1)] $0 < a(x) \le a(0) \quad (x \in \mathbb{R}^N)$.

\item[(a2)] $\displaystyle a\left(-\frac{x}{|x|^2}\right) = a(x) \quad (x \in \mathbb{R}^N)$.

\end{itemize}
Condition {\rm (a2)} means that $a(x)$ is invariant under the map $x \mapsto -\frac{x}{|x|^2}$, 
which is the composition of the inversion $x\mapsto \frac{x}{|x|^2}$ and the reflection $x\mapsto -x$. 
Under (a2), $a(x)$ necessarily satisfies
\begin{equation*}
	\lim_{|x| \to \infty}a(x)=a(0).
\end{equation*}

For $1 \le p < \infty$, we use the following notation.
\begin{align*}
&\|u\|_p = \left(\int_{\mathbb{R}^N} |u|^p \, dx\right)^{\frac{1}{p}} , \quad 
\|u\|_{\infty} = \operatorname*{ess\,sup}_{x \in \mathbb{R}^N} |u(x)| ,\\
&(u,v)_{\mathcal{D}^{1,2}} = \int_{\mathbb{R}^N} \nabla u \cdot \nabla v \, dx , \quad
\|u\|_{\mathcal{D}^{1,2}} = (u,u)_{\mathcal{D}^{1,2}}^{\frac{1}{2}} .
\end{align*}

The equation \eqref{Pa} has a variational structure, and the corresponding functional is defined by
\begin{equation*}
	I(u) = \frac{1}{2} \int_{\mathbb{R}^N} |\nabla u|^2 \, dx - \frac{1}{2^*} \int_{\mathbb{R}^N} a(x) |u|^{2^*} \, dx 
	\quad (u \in \mathcal{D}^{1,2}(\mathbb{R}^N)). 
\end{equation*}
Our goal is to obtain solutions of \eqref{Pa} as critical points of the functional $I$.
As is well known, if the following infimum $c$ is attained, then a suitable rescaling of a minimizer is a critical point of $I$.
\begin{equation}\label{1.1}
	c = \frac{1}{N} \inf_{u \in \mathcal{D}^{1,2}(\mathbb{R}^N) \setminus \{0\}}
	\left(\frac{\int_{\mathbb{R}^N} |\nabla u|^2 \, dx}{\left(\int_{\mathbb{R}^N} a(x) |u|^{2^*} \, dx\right)^{\frac{2}{2^*}}}\right)^{\frac{N}{2}}.
\end{equation}
However, the lack of compactness of the embedding $\mathcal{D}^{1,2}(\mathbb{R}^N) \hookrightarrow L^{2^*}(\mathbb{R}^N)$ makes 
it difficult to establish the compactness of minimizing sequences for this problem. 
Various assumptions of $a(x)$ have been proposed to overcome this lack of compactness (e.g., \cite{AAP, Catrina--Wang, Li1, Li2} and the references therein). 
For example, when $a(x) > 1$ and $\lim_{|x|\to \infty}a(x)=1$, 
compactness of minimizing sequences can be recovered by the concentration-compactness argument \cite{P.L.Lions1, P.L.Lions2}. 
In this argument, it is crucial to show that $c < \frac{1}{N}S^{\frac{N}{2}}$, where $S$ is the best Sobolev constant defined by
\begin{equation*}
	S = \inf_{u \in \mathcal{D}^{1,2}(\mathbb{R}^N) \setminus \{0\}} \frac{\int_{\mathbb{R}^N} |\nabla u|^2 \, dx}{\left(\int_{\mathbb{R}^N} |u|^{2^*} \, dx\right)^{\frac{2}{2^*}}}.
\end{equation*} 
In contrast, assumptions (a1)--(a2) imply $a(x)\le \lim_{|x|\to \infty}a(x)$. 
In the model case where $a(x) < 1$ and $\lim_{|x|\to \infty}a(x)=1$, 
we have $c > \frac{1}{N} S^{\frac{N}{2}}$, and the recovery of compactness becomes substantially more delicate.
In the simplest case where $a(x) \equiv  1$, the equation $(\mathcal{P}_1)$ is invariant under conformal transformations of $\mathbb{R}^N$.
Under stereographic projection, $(\mathcal{P}_1)$ can be transformed into the following equation on the unit sphere $\mathbb{S}^N$:
\begin{equation} \label{1.2}
	-\Delta_{\mathbb{S}^N} u + \mu u = |u|^{2^*-2} u \quad \text{in } \mathbb{S}^N,
\end{equation}
where $\mu > 0$ is a constant. Moreover, this transformed equation \eqref{1.2} remains invariant under conformal transformations of $\mathbb{S}^N$ \cite{Ding}. 
When $a(x) \equiv 1$, the Talenti function
\begin{equation*}
 U(x) = [N (N - 2)]^{\frac{N-2}{4}} \left(\frac{1}{1+|x|^2}\right)^{\frac{N-2}{2}} 
\end{equation*}
attains the infimum in \eqref{1.1}, and under the stereographic projection, $U$ corresponds to a constant solution of equation \eqref{1.2} on $\mathbb{S}^N$ \cite{Talenti}.
Thus, in the case $a(x)\equiv 1$, the loss of compactness can be overcome through the strong conformal symmetry of the problem \cite{Clapp, Ding}. 
This raises the following natural question:
Can we recover compactness assuming only invariance under a single conformal transformation?

Among conformal transformations, we focus on the Kelvin transform $K$. 
It is defined for $u \in C_0^\infty(\mathbb{R}^N \setminus \{0\})$ by
\begin{align}\label{1.3}
(Ku)(x) = \frac{1}{|x|^{N-2}} u \left( -\frac{x}{|x|^2} \right).
\end{align}
Since $C_0^\infty(\R^N \setminus\{0\})$ is dense in $\mathcal{D}^{1,2}(\R^N)$,
$K$ extends uniquely to a bounded linear operator on $\mathcal{D}^{1,2}(\mathbb{R}^N)$.
The Kelvin transform $K$ plays a particularly important role among conformal transformations 
because it exchanges the origin and infinity and has no fixed points even when regarded as a conformal transformation of $\mathbb{S}^N$.
In contrast, although translations on $\mathbb{R}^N$ have no fixed points in $\mathbb{R}^N$, they fix the point at infinity when regarded as conformal transformations of $\mathbb{S}^N$.
This suggests that invariance under the Kelvin transform alone may be sufficient to recover compactness. 

\begin{rmk}
The Kelvin transform is typically defined as
\begin{align*}
	(\widetilde{K}u)(x) = \frac{1}{|x|^{N-2}} u \left( \frac{x}{|x|^2} \right).
\end{align*}
However, in this case, the inversion $x\mapsto \frac{x}{|x|^2}$ fixes every point on the unit sphere, which is inconvenient for our purposes. 
For this reason, we adopt the definition \eqref{1.3} throughout the paper.
\end{rmk}

Motivated by this observation, we focus on the case where (a1)--(a2) hold and the functional $I$ is invariant under the Kelvin transform:
\begin{equation*}
	I(Ku) = I(u) \quad \text{for all } u \in \mathcal{D}^{1,2}(\mathbb{R}^N). 
\end{equation*}
We define the subspace of functions invariant under the Kelvin transform by
\[ \mathcal{D}_K \coloneqq \{u \in \mathcal{D}^{1,2}(\mathbb{R}^N) \mid Ku = u \}. \]
A function $u\in\mathcal D_K$ is called Kelvin--invariant.
We consider the following minimization problem on $\mathcal{D}_K$.
\begin{equation*}
	c_K = \frac{1}{N} \inf_{u \in \mathcal{D}_K \setminus \{0\}} 
	\left(\frac{\int_{\mathbb{R}^N} |\nabla u|^2 \, dx}
	{\left(\int_{\mathbb{R}^N} a(x) |u|^{2^*} \, dx\right)^{\frac{2}{2^*}}}\right)^{\frac{N}{2}}. 
\end{equation*}
If $c_K$ is attained, then a suitable rescaling of a minimizer yields a critical point of $I|_{\mathcal{D}_K}$.
By Palais' principle of symmetric criticality \cite{Palais}, this critical point is also a critical point of $I$.
Our main result is the following.

\begin{thm}
\label{thm 1.2}
Suppose that (a1)--(a2) hold. Moreover, suppose one of the following conditions:
\begin{itemize}

\item[(a3)] $\displaystyle \lim_{|x| \to 0}|x|^{-(N-2)}(a(x)-a(0))=0$. 
\item[(a4)] $\displaystyle \int_{|x| \le 1}\frac{a(x)-2^{-\frac{2}{N-2}}a(0)}{(1+|x|^2)^N} \, dx > 0$.

\end{itemize}
Then $c_K$ is attained and \eqref{Pa} has a Kelvin--invariant positive solution.
\end{thm}

Under (a2), the condition (a3) also implies $\lim_{|x| \to \infty}|x|^{N-2}(a(x)-a(0))=0$, 
so that (a3) imposes a flatness condition at both the origin and infinity. 
On the other hand, condition (a4) arises naturally from the energy estimate used in the proof.
In what follows, we may assume $a(0) = 1$ without loss of generality. 

\medskip

A possible approach to Theorem \ref{thm 1.2} is to transfer the problem to $\mathbb{S}^N$ via stereographic projection 
and derive the result by extending the results of Hebey and Vaugon \cite{Hebey,Hebey--Vaugon} to equations with variable coefficients. 
Instead, we establish the existence result directly in $\R^N$, avoiding the additional geometric analysis on the sphere.
This approach requires a refinement of the abstract profile decomposition developed in \cite{Tintarev07} incorporating the Kelvin transform.
To this end, we first recall the framework of the abstract profile decomposition developed in \cite{Tintarev07}.
Let $H$ be a Hilbert space.
We denote by $\mathcal{B}(H)$ and $\mathcal{U}(H)$ the sets of bounded linear operators and unitary operators on $H$, respectively.

\begin{dfn}[Set of dislocations]
Let $D\subset \mathcal{U}(H)$. 
We say that $D$ is a set of dislocations if every sequence $(g_k) \subset D$ with 
$g_k \not \rightharpoonup 0$ weakly admits a subsequence converging strongly.
\end{dfn}

\begin{dfn}[$D$-weak convergence]
Let $D \subset \mathcal{U}(H)$ and $(u_k) \subset H$. 
We say that $(u_k)$ converges $D$-weakly to $u \in H$ 
if $g_k^*(u_k-u) \rightharpoonup 0$ weakly in $H$ for any sequence $(g_k) \subset D$.
In this case, we write $u_k \overset{D}{\rightharpoonup} u$.
\end{dfn}

\begin{rmk} 
\begin{itemize}

\item[(i)]  Recall that a sequence $(g_k)\subset \mathcal{B}(H)$ converges strongly to $g$ if $g_k u \to gu$ strongly in $H$ for every $u\in H$, 
and converges weakly to $g$ if $g_k u \rightharpoonup gu$ weakly in $H$ for every $u\in H$.

\item[(ii)] For extensions of the profile decomposition to bounded sequences in Banach spaces, we refer to \cite{OkumuraI, OkumuraII}.

\end{itemize}
\end{rmk}

In the classical critical elliptic setting,
$H=\mathcal{D}^{1,2}(\mathbb{R}^N)$, and the set $D$ is given by the group generated by translations and dilations.
In this framework, profile decomposition is used to classify Palais--Smale sequences of the associated energy functional, 
and consequently to recover compactness modulo the symmetries.
In the present setting, however, the Kelvin transform $K$ is not covered by the existing framework of profile decomposition, 
since it does not belong to the dislocation set $D$. 
We therefore refine the abstract profile decomposition by introducing the following compatibility condition between $D$ and $K$.

\begin{dfn}[Compatible with $K$]
Let $D \subset \mathcal{U}(H)$ and $K \in \mathcal{U}(H)$. 
We say that $D$ is compatible with $K$ 
if $g_k^* K g_k \rightharpoonup 0$ weakly 
whenever $(g_k) \subset D$ satisfies $g_k \rightharpoonup  0$ weakly.
\end{dfn}

The following elementary example illustrates this notion.

\begin{rmk}
A simple example of pairs $(D,K)$ satisfying the compatibility condition is given by translations and reflections in $\R^N$.
Let $H=H^1(\mathbb{R}^N)$.
If $D$ is the family of translations $(\tau_y u)(x):=u(x-y)$ for $y\in \R^N$
and $K$ is the reflection with respect to the origin $(Ku)(x):=u(-x)$, then $D$ is compatible with $K$.
In the setting of the present paper, $D$ is the family of translations and dilations, while $K$ is the Kelvin transform. 
In this case, $D$ is also compatible with $K$ (see Lemma~\ref{lem 2.3}).
\end{rmk}

The following theorem is the second main result of this paper.
It provides an abstract refinement of the profile decomposition developed in \cite{Tintarev07} incorporating the Kelvin transform.
Here and below, if $n_0=\infty$, the notation $\{0,1,\dots,n_0\}$ stands for $\mathbb N\cup\{0\}$.

\begin{thm}
\label{thm 1.8}
Let $H$ be a Hilbert space endowed with an inner product $(\cdot,\cdot)$ and the associated norm $\|\cdot\|$.
Let $K \in \mathcal{U}(H)$ satisfy $K^2 = id$, 
and let $D \subset \mathcal{U}(H)$ be a set of dislocations compatible with $K$. 
Set $H_K := \{u \in H \mid Ku = u\}$.
If a sequence $(u_k) \subset H_K$ is bounded in $H$, 
then up to a subsequence, there exist $n_0 \in \mathbb{N} \cup \{0, \infty\}$, $w^{(0)} \in H_K$, $g_k^{(0)} = id$, 
and for each $n  \in \{ 1, 2, \dots, n_0 \}$, there exist $w^{(n)} \in H \setminus \{0\}$ and $ (g_k^{(n)}) \subset D$ such that 
\begin{align}
\label{1.4}
& {g_k^{(n)}}^* u_k \rightharpoonup w^{(n)} \text{ weakly in } H \text{ for all } n\in \{0,1,\dots,n_0\},\\
\label{1.5}
& {g_k^{(n)}}^* g_k^{(m)} \rightharpoonup 0 \text{ weakly for all } n,m\in \{0,1,\dots,n_0\} \text{ with } n\neq m, \\
\label{1.6}
& {g_k^{(n)}}^* K g_k^{(m)} \rightharpoonup 0 \text{ weakly for all } n,m\in \{1,\dots,n_0\},\\
\label{1.7}
& \|w^{(0)}\|^2 + 2 \sum_{n=1}^{n_0} \|w^{(n)}\|^2 \le \limsup_{k\to\infty} \|u_k\|^2.
\end{align}
Moreover, if $n_0 < \infty$, then
\begin{align}
\label{1.8}
u_k - w^{(0)} - \sum_{n=1}^{n_0} (g_k^{(n)}w^{(n)} + Kg_k^{(n)}w^{(n)}) \overset{D}{\rightharpoonup} 0.
\end{align}
\end{thm}

Compared with the classical profile decomposition, 
Theorem \ref{thm 1.8} shows that nontrivial profiles appear in symmetric pairs $g_k^{(n)}w^{(n)}$ and $Kg_k^{(n)}w^{(n)}$.
This observation suggests that the energy level $\frac{2}{N}S^{\frac{N}{2}}$ is the natural compactness threshold in our setting. 
Indeed, a crucial step in the proof of Theorem \ref{thm 1.2} is to establish the estimate $c_K<\frac{2}{N}S^{\frac{N}{2}}$.

\medskip

This paper is organized as follows. 
In Section \ref{section 2}, we prove Theorem \ref{thm 1.8}. 
In Section \ref{section 3}, we establish several preliminary results and derive an energy estimate.
In Section \ref{section 4}, we complete the proof of Theorem \ref{thm 1.2}.

\section{Profile decomposition}
\label{section 2}
In this section, we prove Theorem 1.8, which provides the abstract profile decomposition underlying the proof of Theorem 1.2.

\begin{lem}
\label{lem 2.1}
Let $H$ be a Hilbert space endowed with an inner product $(\cdot,\cdot)$ and the associated norm $\|\cdot\|$. Let $K \in \mathcal{U}(H)$ and $D \subset \mathcal{U}(H)$ be a set of dislocations compatible with $K$. 
Suppose that $(u_k)\subset H_K$ and $(g_k)\subset D$ satisfy
\begin{align*}
& u_k \rightharpoonup 0, \quad g_k^* u_k \rightharpoonup w \neq 0 \quad \text{weakly in } H.
\end{align*}
Then we have
\[ g_k \rightharpoonup 0, \quad g_k^* K g_k \rightharpoonup 0 \quad \text{weakly}. \]
\end{lem}

\begin{proof}
Assume that $g_k \not\rightharpoonup 0$ weakly. Since $D$ is a set of dislocations, there exists $g \in \mathcal{U}(H)$ such that $g_k \to g$ strongly up to a subsequence. Moreover, $g_k^* \to g^*$ strongly because each $g_k$ is unitary. Let $\varphi \in H$. Since $g_k^* u_k \rightharpoonup w$ weakly in $H$ and $g_k^* \varphi \to g^*\varphi$ in $H$, it follows that
\[ (u_k, \varphi) = (g_k^* u_k, g_k^* \varphi) \to (w, g^*\varphi) = (gw, \varphi). \]
Therefore, $u_k \rightharpoonup gw$ weakly in $H$. Since $g$ is unitary and $w \neq 0$, we have $gw \neq 0$, which contradicts the assumption that $u_k \rightharpoonup 0$ weakly in $H$. Hence $g_k \rightharpoonup 0$ weakly.
The second assertion follows immediately, since $D$ is compatible with $K$.
\end{proof}

\begin{proof}[Proof of Theorem \ref{thm 1.8}]
Since $(u_k)$ is bounded in $H$, up to a subsequence, there exists $w^{(0)} \in H$ such that
\[ u_k \rightharpoonup w^{(0)} \quad \text{ weakly in } H. \]
Since $H_K$ is a closed subspace of $H$, it is also weakly closed. Therefore, $w^{(0)} \in H_K$. Moreover, we set $g_k^{(0)} = id$. We divide the proof into several steps.

\medskip

\noindent
\textbf{Step 1.} \textit{Set 
\[ r_k^{(1)} = u_k - w^{(0)}. \]
Then, either (A) or (B) holds.
\begin{itemize}
\item[(A)] If $r_k^{(1)} \overset{D}{\rightharpoonup} 0$, then \eqref{1.4}--\eqref{1.5} and \eqref{1.8} hold with $n_0 = 0$.
\item[(B)] If $r_k^{(1)} \not\overset{D}{\rightharpoonup} 0$, then up to a subsequence, there exist $w^{(1)} \in H \setminus \{0\}$ and $(g_k^{(1)}) \subset D$ such that 
\begin{align}
\label{2.1}
& {g_k^{(1)}}^* u_k \rightharpoonup w^{(1)} \quad \text{weakly in } H,\\
\label{2.2}
& g_k^{(1)} \rightharpoonup 0 \quad \text{weakly}.\\
\label{2.3}
& {g_k^{(1)}}^* K g_k^{(1)} \rightharpoonup 0\quad \text{weakly}.
\end{align}
\end{itemize}
}

\begin{proof}[Proof of Step 1]
Since (A) is obvious, we only prove (B).
Note that $r_k^{(1)}\in H_K$ and $r_k^{(1)} \rightharpoonup 0$ weakly in $H$. 
Suppose that $r_k^{(1)} \not\overset{D}{\rightharpoonup} 0$. 
Then, there exists $(g_k^{(1)}) \subset D$ such that
\begin{equation*}
 {g_k^{(1)}}^* r_k^{(1)} \not\rightharpoonup 0 \quad \text{weakly in } H. 
\end{equation*}
Since $({g_k^{(1)}}^* r_k^{(1)} )$ is bounded in $H$, up to subsequence, there exists $w^{(1)} \in H \setminus \{0\}$ such that
\begin{equation*}
 {g_k^{(1)}}^* r_k^{(1)} \rightharpoonup w^{(1)} \quad \text{weakly in } H. 
\end{equation*}
Thus, applying Lemma \ref{lem 2.1} with $u_k = r_k^{(1)}$ and $g_k = g_k^{(1)}$, we obtain \eqref{2.2}--\eqref{2.3}.
Moreover, 
\[ {g_k^{(1)}}^* u_k = {g_k^{(1)}}^* r_k^{(1)} + {g_k^{(1)}}^* w^{(0)} \rightharpoonup w^{(1)} \quad \text{weakly in } H, \]
which yields \eqref{2.1}.
\end{proof}

We now argue inductively.

\medskip

\noindent
\textbf{Step 2.} \textit{For $l \ge 1$, we assume that there exist $w^{(0)} \in H_K$, $g_k^{(0)} = id$, 
and for each $n\in \{1,\dots,n_0\}$, there exists $w^{(n)} \in H \setminus \{0\}$ and $(g_k^{(n)}) \subset D$ such that
\begin{align}
\label{2.4}
& {g_k^{(n)}}^* u_k \rightharpoonup w^{(n)} \text{ weakly in } H \text{ for all } n\in \{0,1,\cdots,l\}.\\
\label{2.5}
& {g_k^{(n)}}^* g_k^{(m)} \rightharpoonup 0\text{ weakly for all } n,m\in \{0,1,\dots,l\} \text{ with } n\neq m,  \\
\label{2.6}
& {g_k^{(n)}}^* K g_k^{(m)} \rightharpoonup 0 \text{ weakly for all } n,m\in \{1,\dots,l\}.
\end{align}
Set
\begin{equation*}
	r_k^{(l+1)} = u_k - w^{(0)} - \sum_{n=1}^l (g_k^{(n)}w^{(n)} + Kg_k^{(n)} w^{(n)}). 
\end{equation*}
Then, either (A) or (B) holds.
\begin{itemize}
\item[(A)] If $r_k^{(l+1)} \overset{D}{\rightharpoonup} 0$, then \eqref{1.4}--\eqref{1.6} and \eqref{1.8} hold with $n_0 = l$.
\item[(B)] If $r_k^{(l+1)} \not\overset{D}{\rightharpoonup} 0$, then up to a subsequence, 
there exist $w^{(l+1)} \in H \setminus \{0\}$ and $(g_k^{(l+1)}) \subset D$ such that
\begin{align}
\label{2.7}
& {g_k^{(l+1)}}^* u_k \rightharpoonup w^{(l+1)} \text{ weakly in } H,\\
\label{2.8}
& {g_k^{(n)}}^* g_k^{(l+1)} \rightharpoonup 0 \text{ weakly for all } n\in\{0,1,\dots,l\},\\
\label{2.9}
& {g_k^{(n)}}^*  K g_k^{(l+1)} \rightharpoonup 0 \text{ weakly for all } n\in\{1,\dots,l+1\}.
\end{align}
\end{itemize}
}
\begin{proof}[Proof of Step 2]
Since (A) is obvious, we only prove (B). 
Using \eqref{2.4}--\eqref{2.6}, we have
\begin{equation}\label{2.10}
{g_k^{(n)}}^* r_k^{(l+1)} \rightharpoonup 0 \quad \text{weakly in } H \text{ for all } n \in \{0, 1, \dots, l\}. 
\end{equation}
Furthermore, since $u_k, w^{(0)} \in H_K$ and $K^2 = id$, we have $r_k^{(l+1)}\in H_K$.
Suppose that $r_k^{(l+1)} \not\overset{D}{\rightharpoonup} 0$. Then, there exists $(g_k^{(l+1)}) \subset D$ such that
\[ {g_k^{(l+1)}}^* r_k^{(l+1)} \not\rightharpoonup 0 \quad \text{weakly in } H. \]
Since $({g_k^{(l+1)}}^* r_k^{(l+1)} )$ is bounded in $H$, up to subsequence, there exists $w^{(l+1)} \in H \setminus \{0\}$ such that
\begin{equation}\label{2.11}
{g_k^{(l+1)}}^* r_k^{(l+1)} \rightharpoonup w^{(l+1)} \quad \text{weakly in } H. 
\end{equation}
Thus, applying Lemma \ref{lem 2.1} with $u_k = r_k^{(l+1)}$ and $g_k={g_k^{(l+1)}}$, we obtain
\begin{equation*}
g_k^{(l+1)} \rightharpoonup 0, \quad {g_k^{(l+1)}}^* Kg_k^{(l+1)} \rightharpoonup 0 \quad \text{weakly}.
\end{equation*}
Furthermore, \eqref{2.11} can be rewritten as
\begin{align}
& \left[{g_k^{(n)}}^*{g_k^{(l+1)}}\right]^* {g_k^{(n)}}^* r_k^{(l+1)} \rightharpoonup w^{(l+1)} \text{ weakly in } H \text{ for all } n\in\{0,1,\dots,l\}.\label{2.12}\\
& \left[{g_k^{(n)}}^* K{g_k^{(l+1)}}\right]^* {g_k^{(n)}}^* r_k^{(l+1)} \rightharpoonup w^{(l+1)} \text{ weakly in } H \text{ for all } n\in\{1,\dots,l\}.\label{2.13}
\end{align}
By \eqref{2.12}, \eqref{2.13}, and \eqref{2.10}, applying Lemma \ref{lem 2.1} with $u_k = {g_k^{(n)}}^* r_k^{(l+1)}$ and $g_k= {g_k^{(n)}}^* {g_k^{(l+1)}}$ or ${g_k^{(n)}}^* K{g_k^{(l+1)}}$,
we obtain \eqref{2.8} and  \eqref{2.9}. Moreover, 
\[ {g_k^{(l+1)}}^* u_k = {g_k^{(l+1)}}^* r_k^{(l+1)} + \sum_{n=1}^l ({g_k^{(l+1)}}^* g_k^{(n)}w^{(n)} + {g_k^{(l+1)}}^* Kg_k^{(n)} w^{(n)}) \rightharpoonup w^{(l+1)} \quad \text{weakly in } H, \]
which yields \eqref{2.7}. 
\end{proof}

If Step 2 (A) holds for some $l$, we proceed to Step 3 by setting $n_0=l$.
If Step 2 (B) holds for some $l$, using \eqref{2.8} and \eqref{2.9}, we have
\begin{align*}
& {{g_k^{(n)}}^* g_k^{(l+1)}} = ({g_k^{(l+1)}}^* g_k^{(n)})^* \rightharpoonup 0 \quad \text{ weakly for all } n \in \{0, 1, \dots, l\},\\
& {{g_k^{(n)}}^*  K g_k^{(l+1)}} = ({g_k^{(l+1)}}^* K g_k^{(n)})^* \rightharpoonup 0 \quad \text{ weakly for all } n \in \{1, \dots, l+1\}.
\end{align*}
Thus, the inductive assumptions in Step 2 hold for $l+1$ and repeat Step 2.
If Step 2 (B) holds for all $l \in \N$, then, by a diagonal argument, up to a subsequence, \eqref{2.7}--\eqref{2.9} hold for all $l \in \N$.
Therefore, \eqref{1.4}--\eqref{1.6} hold with $n_0=\infty$, and we proceed to Step 3.

\medskip

\noindent
\textbf{Step 3.} \textit{For the $n_0$ determined above, \eqref{1.7} holds. }
\begin{proof}[Proof of Step 3]
Fix $M \in \mathbb{N}$ such that $M \le n_0$. 
Set
\begin{equation*}
v_k=w^{(0)}+\sum_{n=1}^M \left(g_k^{(n)}w^{(n)}+Kg_k^{(n)}w^{(n)}\right).
\end{equation*}
Since $K$ and $g_k^{(n)}$ are unitary, using \eqref{1.4}--\eqref{1.6}, we obtain
\begin{align*}
\|v_k\|^2 & = \|w^{(0)}\|^2 + 2\sum_{n=1}^M \|w^{(n)}\|^2 + o(1), \\
(u_k,v_k) & = \|w^{(0)}\|^2 + 2\sum_{n=1}^M \|w^{(n)}\|^2 + o(1),
\end{align*}
Therefore,
\begin{align*}
0 \le \|u_k-v_k\|^2 
& =\|u_k\|^2 - 2(u_k,v_k) + \|v_k\|^2 \\
& = \|u_k\|^2-\|w^{(0)}\|^2-2\sum_{n=1}^M \|w^{(n)}\|^2+o(1).
\end{align*}
Taking $\limsup$ as $k\to\infty$, we obtain
\begin{equation*}
\|w^{(0)}\|^2 + 2\sum_{n=1}^{M} \|w^{(n)}\|^2 \le \limsup_{k\to\infty} \|u_k\|^2.
\end{equation*}
If $n_0 < \infty$, we obtain \eqref{1.7} by setting $M = n_0$. If $n_0 = \infty$, \eqref{1.7} follows by letting $M \to \infty$.
\end{proof}
Steps 1--3 complete the proof.
\end{proof}

We conclude this section by describing the framework to prove Theorem \ref{thm 1.2} applying Theorem \ref{thm 1.8}. For $y \in \mathbb{R}^N$ and $\lambda \in \R$, 
we define translation operator $\tau_y$ 
and dilation operator $\delta_\lambda$ on $\mathcal{D}^{1,2}(\R^N)$ as follows:
\begin{equation*}
	(\tau_y u)(x) \coloneqq u (x - y), \quad 
	(\delta_\lambda u)(x) \coloneqq 2^{\frac{N-2}{2}\lambda} u(2^\lambda x).
\end{equation*}
	Furthermore, the translation--dilation operator ${T}_{y,\lambda}$ is given by ${T}_{y,\lambda} := \tau_y \circ \delta_\lambda$, that is,
\begin{equation*}
	({T}_{y,\lambda}u)(x) = 2^{\frac{N-2}{2}\lambda} u(2^{\lambda}(x-y)).
\end{equation*}
A direct computation shows that ${T}^{-1}_{y,\lambda} = T_{-2^{\lambda} y,-\lambda}$ for all $y \in \mathbb{R}^N$ and $\lambda \in \R$. We set 
\[ G \coloneqq \{{T}_{y,\lambda} \mid y \in \mathbb{R}^N, \ \lambda \in \mathbb{R} \}.\]
Then, $G$ consists of unitary operators on $\mathcal{D}^{1,2}(\mathbb{R}^N)$ and isometries on $L^{2^*}(\mathbb{R}^N)$. 
Moreover, a direct calculation shows that
\begin{equation*}
\delta_\lambda \circ \tau_{y} = \tau_{2^{-\lambda} y} \circ \delta_\lambda \quad \text{ for all } y \in \mathbb{R}^N \text{ and } \lambda \in \mathbb{R}.
\end{equation*}
Therefore, setting $S_{\lambda,y} = \delta_\lambda \circ \tau_y$ for $\lambda \in \mathbb{R}$ and $y \in \mathbb{R}^N$, we have
\begin{equation*}
\label{200}
G = \{S_{\lambda,y} \mid \lambda \in \mathbb{R}, y \in \mathbb{R}^N\}.
\end{equation*}

We introduce two lemmas. The first lemma shows that $G$ is a set of dislocations.

\begin{lem}
\label{lem 2.2}
For $(y_k) \subset \mathbb{R}^N$ and $(\lambda_k) \subset \R$, the following are equivalent. 
\begin{itemize}
\item[(i)] ${T}_{y_k,\lambda_k} \rightharpoonup 0$ weakly.
\item[(ii)] $|y_k| + |\lambda_k| \to \infty.$
\end{itemize}
In particular, $G$ is a set of dislocations.
\end{lem}

\begin{proof}
First, suppose that $|y_k| + |\lambda_k| \not\to \infty$. Then, up to a subsequence, there exist $y \in \mathbb{R}^N$ and $\lambda \in \R$ such that $y_k \to y$ and $\lambda_k \to \lambda$. 
Choose $u \in C_0^\infty(\mathbb{R}^N) \setminus \{0\}$. 
Then
\begin{align*}
(T_{y_k,\lambda_k} u, T_{y,\lambda}u)_{\mathcal{D}^{1,2}} \to (u, u)_{\mathcal{D}^{1,2}} = \|u\|_{\mathcal{D}^{1,2}}^2 > 0.
\end{align*}
Therefore, ${T}_{y_k,\lambda_k} \not\rightharpoonup 0$ weakly.

Conversely, assume that $|y_k| + |\lambda_k| \to \infty$. 
By the density of $C_0^\infty(\R^N)$ in $\mathcal{D}^{1,2}(\R^N)$, it suffices to show that $(T_{y_k,\lambda_k}u, v)_{\mathcal{D}^{1,2}} \to 0$ for all $u, v \in C_0^\infty(\mathbb{R}^N)$. 
Passing to a subsequence, we may assume that either $|y_k| \to \infty$ and $(\lambda_k)$ is bounded, or $|\lambda_k|\to \infty$.
Let $u, v \in C_0^\infty(\mathbb{R}^N)$. 
If $|y_k| \to \infty$ and $(\lambda_k)$ is bounded, 
then ${\rm supp}(T_{y_k,\lambda_k}u) \cap {\rm supp}(v) = \emptyset$ for all sufficiently large $k \in \mathbb{N}$. 
Thus, $(T_{y_k,\lambda_k}u, v)_{\mathcal{D}^{1,2}} \to 0$. 
Suppose next $|\lambda_k| \to \infty$. 
Then,
\begin{equation*}
\begin{split}
(T_{y_k,\lambda_k}u, v)_{\mathcal{D}^{1,2}} &= 2^{\frac{N}{2}\lambda_k} \int_{\mathbb{R}^N} (\nabla u)(2^{\lambda_k}( x-y_k)) \cdot \nabla v(x) \, dx\\
&= 2^{-\frac{N}{2}\lambda_k} \int_{\mathbb{R}^N} \nabla u(x) \cdot (\nabla v)(2^{-\lambda_k}x+y_k) \, dx.
\end{split}
\end{equation*}
Therefore, $(T_{y_k,\lambda_k}u, v)_{\mathcal{D}^{1,2}} \to 0$. 
Since the above argument applies to every subsequence, the original sequence also converges to $0$.

Finally, we show that $G$ is a set of dislocations. Let $(T_{y_k,\lambda_k}) \subset G$ be a sequence such that $T_{y_k,\lambda_k} \not\rightharpoonup 0$ weakly. 
By the equivalence of {\rm (i)} and {\rm (ii)}, we have $|y_k| + |\lambda_k| \not\to \infty$. Up to a subsequence, we may assume $y_k \to y$ and $\lambda_k \to \lambda$ for some $y \in \R^N, \lambda \in \R$. 
Then, for $u \in C_0^\infty(\mathbb{R}^N)$, it follows that $T_{y_k,\lambda_k}u \to T_{y,\lambda}u$ in $\mathcal D^{1,2}(\mathbb R^N)$.
Since $C_0^\infty(\R^N)$ is dense in $\mathcal{D}^{1,2}(\R^N)$, we have
$T_{y_k,\lambda_k} \to T_{y,\lambda}$ strongly, which implies that $G$ is a set of dislocations.
\end{proof}

The second lemma proves that $D$ is compatible with the Kelvin transform on $\mathcal{D}^{1,2}(\mathbb{R}^N)$. 

\begin{lem}
\label{lem 2.3}
$D$ is compatible with the Kelvin transform $K$ on $\mathcal{D}^{1,2}(\mathbb{R}^N)$.
\end{lem}

\begin{proof}
Let $(S_{\lambda_k,y_k}) \subset G$ such that $S_{\lambda_k,y_k} \rightharpoonup 0$ weakly. By Lemma \ref{lem 2.2} and $S_{\lambda_k,y_k} = T_{2^{-\lambda_k}y_k, \lambda_k}$, it follows that $|2^{-\lambda_k}y_k| + |\lambda_k| \to \infty$, which yields $|y_k| + |\lambda_k| \to \infty$.
For all $u, v \in \mathcal{D}^{1,2}(\R^N)$, we have 
\begin{align*}
(S_{\lambda_k,y_k}^* K S_{\lambda_k,y_k} u, v)_{\mathcal{D}^{1,2}} &= (K S_{\lambda_k,y_k} u, S_{\lambda_k,y_k}  v)_{\mathcal{D}^{1,2}}.
\end{align*}
Since $C_0^\infty(\mathbb{R}^N \setminus \{0\})$ is dense in $\mathcal{D}^{1,2}(\mathbb{R}^N)$, 
it suffices to show that $(K S_{\lambda_k,y_k} u, S_{\lambda_k,y_k} v)_{\mathcal{D}^{1,2}}\to 0$ for all $u, v \in C_0^\infty(\mathbb{R}^N \setminus \{0\})$.
Let $u, v \in C_0^\infty(\mathbb{R}^N \setminus \{0\})$.
Passing to a subsequence, we may assume that either $|y_k| \to \infty$, or $|\lambda_k|\to \infty$ and $(y_k)$ is bounded.

First, suppose $|y_k| \to \infty$. Then, for all sufficiently large $k\in\mathbb{N}$, we have 
\begin{equation*}
{\rm supp}(\tau_{y_k}u), \ {\rm supp}(\tau_{y_k}v) \subset \{x \in \mathbb{R}^N \mid x \cdot y_k > 0\}.
\end{equation*}
Since $S_{\lambda_k,y_k} = \delta_{\lambda_k} \circ \tau_{y_k}$ and  $\delta_{\lambda_k}$ is a dilation with respect to the origin in $\R^N$, it follows that
\begin{equation}\label{2.14}
\operatorname{supp}(S_{\lambda_k,y_k}u), \operatorname{supp}(S_{\lambda_k,y_k}v) \subset \{x \in \mathbb{R}^N \mid x \cdot y_k > 0\}.
\end{equation}
Moreover, we obtain
\begin{equation}
\label{2.15}
\operatorname{supp}(KS_{\lambda_k,y_k}u)\subset \{x \in \mathbb{R}^N \mid x \cdot y_k < 0\}.
\end{equation}
Indeed, let $x \in \operatorname{supp}(KS_{\lambda_k,y_k}u)$. By the definition of the Kelvin transform, we have $-\frac{x}{|x|^2} \in \operatorname{supp}(S_{\lambda_k,y_k}u)$. 
Hence,  \eqref{2.14} implies $x \cdot y_k<0$. 
Thus, \eqref{2.15} holds. 
Therefore, we obtain $\operatorname{supp}(KS_{\lambda_k,y_k}u) \cap \operatorname{supp}(S_{\lambda_k,y_k}v) = \emptyset$ for all sufficiently large $k \in \mathbb{N}$, which implies $(KS_{\lambda_k,y_k}u, S_{\lambda_k,y_k} v)_{\mathcal{D}^{1,2}}\to 0$.

Next, suppose $|\lambda_k| \to \infty$ and $(y_k)$ is bounded. Then, it suffices to show that
\begin{equation}
\label{2.16}
(K\delta_{\lambda_k}u, \delta_{\lambda_k} v)_{\mathcal{D}^{1,2}}\to 0.
\end{equation}
In fact, suppose for contradiction that $(KS_{\lambda_k,y_k} u, S_{\lambda_k,y_k} v)_{\mathcal{D}^{1,2}} \not \to 0$. 
Then, passing to a subsequence, we may assume $y_k \to y$ for some $y \in \R^N$. Moreover, up to a further subsequence, there exists $C > 0$ such that
\begin{equation*}
\lim_{k \to \infty} |(KS_{\lambda_k,y_k}u, S_{\lambda_k,y_k} v)_{\mathcal{D}^{1,2}}| \ge C.
\end{equation*}
Since $S_{\lambda_k,y_k} = \delta_{\lambda_k} \circ \tau_{y_k}$ and $\tau_{y_k} \to \tau_y$ strongly, we have
\begin{equation*}
\lim_{k \to \infty} |(K\delta_{\lambda_k}(\tau_y u), \delta_{\lambda_k}(\tau_y v))_{\mathcal{D}^{1,2}}| \ge C,
\end{equation*}
which contradicts \eqref{2.16}. Now, a direct calculation shows that $K \delta_{\lambda_k} = \delta_{-\lambda_k} K$. 
Hence, we have
\begin{equation*}
(K\delta_{\lambda_k}u, \delta_{\lambda_k} v)_{\mathcal{D}^{1,2}} = (\delta_{-\lambda_k} Ku, \delta_{\lambda_k} v)_{\mathcal{D}^{1,2}} = (Ku, \delta_{2\lambda_k} v)_{\mathcal{D}^{1,2}}.
\end{equation*}
Lemma \ref{lem 2.2} yields $(Ku, \delta_{2\lambda_k} v)_{\mathcal{D}^{1,2}} \to 0$
as $|\lambda_k| \to \infty$.
Thus, \eqref{2.16} holds. 

Consequently, in either cases, we proved $(KS_{\lambda_k,y_k}u, S_{\lambda_k,y_k} v)_{\mathcal{D}^{1,2}}\to 0$.
Since the above argument applies to every subsequence, the original sequence also converges to $0$.
Therefore, the proof is complete.
\end{proof}

By Lemmas \ref{lem 2.2} and \ref{lem 2.3}, the assumptions of Theorem \ref{thm 1.8} are satisfied with $H=\mathcal{D}^{1,2}(\mathbb{R}^N)$, $D=G$, 
and $K$ equal to the Kelvin transform on $\mathcal{D}^{1,2}(\mathbb{R}^N)$. 
Therefore, every bounded sequence $(u_k) \subset \mathcal{D}_K$ admits a profile decomposition. 
This fact is used in the proof of Theorem \ref{thm 1.2}.

\section{Preliminaries and energy estimate}
\label{section 3}
In this section, we derive the energy estimate for the level $c_K$ required for the proof of Theorem \ref{thm 1.2}. 
A key ingredient in our argument is the invariance of the functional $I$ under the Kelvin transform. 
We therefore begin by recalling several properties of the Kelvin transform.

\begin{lem}
\label{lem 3.1}
The Kelvin transform $K$ on $\mathcal{D}^{1,2}(\mathbb{R}^N)$ satisfies the following properties.
\begin{itemize}
\item[(i)] $K^2 = id$.
\item[(ii)] $\|Ku\|_{\mathcal{D}^{1,2}} = \|u\|_{\mathcal{D}^{1,2}}$.
\item[(iii)] If (a2) holds, we have 
$\displaystyle \int_{\mathbb{R}^N} a(x) |Ku|^{2^*} \, dx = \int_{\mathbb{R}^N} a(x) |u|^{2^*} \, dx.$
\end{itemize}
In particular, $K$ is a self-adjoint unitary operator on $\mathcal{D}^{1,2}(\mathbb{R}^N)$, and an isometry on $L^{2^*}(\mathbb{R}^N)$.
\end{lem}

\begin{proof}
Since $C_0^\infty(\mathbb{R}^N \setminus \{0\})$ is dense in $\mathcal{D}^{1,2}(\mathbb{R}^N)$,
it suffices to show the results for functions in $C_0^\infty(\mathbb{R}^N \setminus \{0\})$. 
For $u \in C_0^\infty(\mathbb{R}^N \setminus \{0\})$,
\begin{align*}
(K^2u)(x) & =\frac{1}{|x|^{N-2}}(Ku)\left(-\frac{x}{|x|^2}\right)
= \frac{1}{|x|^{N-2}} \frac{1}{\left|-\frac{x}{|x|^2}\right|^{N-2}} u\left(-\frac{-\frac{x}{|x|^2}}{\left|-\frac{x}{|x|^2}\right|^2}\right)
= u(x).
\end{align*}
Thus, (i) holds. 
A direct calculation shows that $\Delta(Kp)=K(|x|^4\Delta p)$ for every polynomial $p$ (see \cite[Proposition 4.6]{ABR}). 
Since polynomials are locally dense in $C^2$, it follows that
\begin{equation*}
\Delta(Ku)=K(|x|^4\Delta u) = {|x|^{-N-2}}(\Delta u)\left(-\frac{x}{|x|^2}\right)
\end{equation*}
for every $u\in C_0^\infty(\mathbb{R}^N\setminus\{0\})$.
Hence, for $u,v\in C_0^\infty(\mathbb{R}^N\setminus\{0\})$, we have
\begin{align*}
(Ku, Kv)_{\mathcal{D}^{1,2}} 
& = \int_{\R^N} \nabla(Ku) \cdot \nabla(Kv)\, dx \\
& = -\int_{\R^N} \Delta(Ku) Kv\, dx \\
& = -\int_{\mathbb{R}^N} |x|^{-2N} (\Delta u)\left(-\frac{x}{|x|^2}\right) v\left(-\frac{x}{|x|^2}\right) \, dx.
\end{align*}
By the change of variables $y = - \frac{x}{|x|^2}$, whose Jacobian determinant is $|y|^{-2N}$, we obtain
\begin{equation*}
(Ku, Kv)_{\mathcal{D}^{1,2}} 
= -\int_{\mathbb{R}^N} |y|^{2N} (\Delta u)(y) v(y) |y|^{-2N} \, dy
= -\int_{\R^N} (\Delta u) v\, dx \\
= (u, v)_{\mathcal{D}^{1,2}}.
\end{equation*}
Therefore, $K$ is an isometry on $\mathcal{D}^{1,2}(\mathbb{R}^N)$. 
In particular, (ii) holds. 
By $K^2=id$ and the isometry of $K$, we obtain $K^{-1} = K = K^*$. 
Hence, $K$ is a self-adjoint unitary operator on $\mathcal{D}^{1,2}(\mathbb{R}^N)$. 
Moreover,  the change of variables $x=-\frac{y}{|y|^2}$ yields
\begin{equation*}
\int_{\mathbb{R}^N} a(x) |Ku|^{2^*} \, dx 
= \int_{\mathbb{R}^N} a(x) \frac{1}{|x|^{2N}}\left|u\left(-\frac{x}{|x|^2}\right)\right|^{2^*} \, dx 
= \int_{\mathbb{R}^N} a\left(-\frac{y}{|y|^2}\right) |u|^{2^*} \, dy.
\end{equation*}
By (a2), (iii) holds. 
Taking $a\equiv1$ in (iii), we conclude that $K$ is an isometry on $L^{2^*}(\mathbb{R}^N)$.
\end{proof}

By Lemma \ref{lem 3.1}, if (a2) holds, then we have 
\begin{equation*}
I(Ku) = I(u) \quad \text{ for all } u \in \mathcal{D}^{1,2}(\mathbb{R}^N),
\end{equation*} 
that is, $I$ is invariant under the Kelvin transform.

Next, we establish an energy estimate for $c_K$.
Energy estimates are often essential for obtaining compactness through concentration--compactness arguments as seen, for example, in \cite{Brezis--Nirenberg, P.L.Lions1, P.L.Lions2}. 
For frameworks imposing symmetry assumptions, we refer to \cite{Adachi, Bahri--Li, Clapp, Hirata08}.
The following proposition provides the desired estimate for $c_K$.

\begin{prop}
\label{prop 3.2}
Suppose that (a1)--(a2) hold and that either (a3) or (a4) holds.
Then we have $c_K < \frac{2}{N}S^{\frac{N}{2}}$.
\end{prop}

To prove Proposition \ref{prop 3.2}, we employ a family of Talenti functions $(U_{\varepsilon})_{\varepsilon>0}$ defined by 
\begin{equation*}
	U_\varepsilon(x) = \varepsilon^{-\frac{N-2}{2}}U(\varepsilon^{-1}x)
	= [N (N - 2)]^{\frac{N-2}{4}} \left(\frac{1}{\varepsilon^2 + |x|^2}\right)^{\frac{N-2}{2}}. 
\end{equation*}  
Then, $U_\varepsilon$ attains $S$ and is a positive solution of $-\Delta u = u^{2^*-1}$ in $\R^N$.  
Therefore, $U_{\varepsilon}$ satisfies
\begin{align}
\label{3.1}
& \int_{\mathbb{R}^N} \nabla U_{\varepsilon} \cdot \nabla \varphi \, dx = \int_{\mathbb{R}^N} U_{\varepsilon}^{2^*-1}\varphi \, dx \quad \text{for all } \varphi \in \mathcal{D}^{1,2}(\mathbb{R}^N), \\
\label{3.2}
& \int_{\mathbb{R}^N} U_{\varepsilon}^{2^*} \, dx = \int_{\mathbb{R}^N} |\nabla U_{\varepsilon}|^2 \, dx = S^{\frac{N}{2}}.
\end{align}
Moreover, a direct calculation shows that 
\begin{equation}\label{3.3}
	KU_\varepsilon = U_{\frac{1}{\varepsilon}} \quad \text{ for } \varepsilon>0.
\end{equation}

\begin{lem}
\label{lem 3.3}
Suppose (a1)--(a3). Then, as $\varepsilon \to 0^+$, the following hold:
\begin{itemize}
    \item[(i)] $\displaystyle \int_{\mathbb{R}^N} a(x) U_\varepsilon^{2^*} \, dx = S^{\frac{N}{2}} + o(\varepsilon^{N-2}),$
    \item[(ii)] $\displaystyle \int_{\mathbb{R}^N} a(x) U_\varepsilon^{2^*-1} U_{\frac{1}{\varepsilon}} \, dx = A \varepsilon^{N-2} + o(\varepsilon^{N-2})$, where $A = U(0)\|U\|_{2^*-1}^{2^* -1}$.
\end{itemize}
\end{lem}

\begin{proof}
We first prove (i).
By \eqref{3.2}, we have
\begin{align}
\notag
\int_{\mathbb{R}^N} a(x) U_\varepsilon^{2^*} \, dx &= \int_{\mathbb{R}^N} U_\varepsilon^{2^*} \, dx +  \int_{\mathbb{R}^N} (a(x)-1) U_\varepsilon^{2^*} \, dx \\
\label{3.4}
&= S^{\frac{N}{2}} + \int_{\mathbb{R}^N} (a(x)-1) U_\varepsilon^{2^*} \, dx.
\end{align}
Let $\eta > 0$. From (a3), there exists $r > 0$ such that 
\begin{equation}
\label{3.5}
|x|^{-(N-2)} |a(x)-1| < \eta \quad (|x| \le r).
\end{equation}
Moreover, there exists $C > 0$ such that
\begin{equation}
\label{3.6}
|x|^{-(N-2)} |a(x)-1| \le C \quad (|x| \ge r).
\end{equation}
Using \eqref{3.5} and \eqref{3.6}, we have 
\begin{align*}
\left| \int_{\mathbb{R}^N} (a(x)-1) U_\varepsilon^{2^*} \, dx \right|
& \le \left(\int_{|x| \le r} + \int_{|x| \ge r}\right) |x|^{-(N-2)} |a(x)-1| |x|^{N-2} U_\varepsilon^{2^*} \, dx \\
&\le \eta  \int_{|x| \le r} |x|^{N-2} U_\varepsilon^{2^*} \, dx + C \int_{|x| \ge r} |x|^{N-2} U_\varepsilon^{2^*} \, dx.
\end{align*}
By the change of variables $x=\varepsilon y$, it follows that
\begin{align*}
\left| \int_{\mathbb{R}^N} (a(x)-1) U_\varepsilon^{2^*} \, dx \right|
&\le \eta \varepsilon^{N-2} \int_{|y| \le \frac{r}{\varepsilon}} |y|^{N-2} U^{2^*} \, dy 
+ C \varepsilon^{N-2}  \int_{|y| \ge \frac{r}{\varepsilon}} |y|^{N-2} U^{2^*} \, dy\\
&\le \eta  \varepsilon^{N-2} \int_{|y| \le \frac{r}{\varepsilon}} |y|^{N-2} U^{2^*} \, dy + o(\varepsilon^{N-2}).
\end{align*}
Since $|y|^{N-2} U^{2^*}\in L^1(\R^N)$ and $\eta > 0$ is arbitrary, we have
\begin{align}
\label{3.7}
\int_{\mathbb{R}^N} (a(x)-1) U_\varepsilon^{2^*} \, dx & = o(\varepsilon^{N-2}).
\end{align}
By \eqref{3.4} and \eqref{3.7}, (i) holds. 
We next prove (ii).
By the change of variables $x=\varepsilon y$, we have
\begin{align*}
\int_{\mathbb{R}^N} a(x) U_\varepsilon^{2^*-1} U_{\frac{1}{\varepsilon}} \, dx 
& = \int_{\mathbb{R}^N} a(x)\varepsilon^{-\frac{N+2}{2}}U(\varepsilon^{-1}x)^{2^*-1}\varepsilon^{\frac{N-2}{2}}U(\varepsilon x) \, dx\\
&= \varepsilon^{N-2} \int_{\mathbb{R}^N} a(\varepsilon y)U(y)^{2^*-1}U(\varepsilon^2 y) \, dy.
\end{align*}
Since $a(\varepsilon y)\to a(0)=1$ and $U(\varepsilon^2 y)\to U(0)$ pointwise as $\varepsilon \to 0$, we deduce
\[ \int_{\mathbb{R}^N} a(x) U_\varepsilon^{2^*-1} U_{\frac{1}{\varepsilon}} \, dx = U(0)\left(\int_{\mathbb{R}^N} U^{2^*-1} \, dx\right)\varepsilon^{N-2} + o(\varepsilon^{N-2}). \]
Therefore, (ii) holds.
\end{proof}

The following inequality will also be used.

\begin{lem}
\label{lem 3.4}
Let $p > 2$. Then, there exists a constant $C_p > \frac{p}{2}$ such that, for all $a, b \ge 0$, the following inequality holds:
\begin{align}
\label{3.8}
(a+b)^{p} &\ge a^{p} + b^{p} + C_p(a^{p-1}b + a b^{p-1}).
\end{align}
\end{lem}

\begin{proof}
When $a=0$ or $b = 0$, the inequality is trivial. 
Thus, it suffices to prove \eqref{3.8} for all $a, b > 0$. We first show the following strict inequality:
\begin{equation}
\label{3.9}
(a+b)^{p} > a^{p} + b^{p} + \frac{p}{2}(a^{p-1}b + a b^{p-1}) \quad \text{ for all } a,b>0.
\end{equation}
From $p > 2$, for all $a, b > 0$, we have
\begin{align*}
(a+b)^p - a^p - b^p - p a^{p-1}b &= p \int_{0}^{b} \left[ (a+t)^{p-1} - t^{p-1} - a^{p-1} \right] \, dt\\
&= p(p-1) \int_{0}^{b} \left( \int_{0}^{t} \left[ (a+s)^{p-2} - s^{p-2} \right] \, ds \right) \, dt > 0.
\end{align*}
By interchanging the roles of $a$ and $b$, we also obtain
\[ (a+b)^p - a^p - b^p - p b^{p-1} a > 0 \quad \text{ for all } a,b>0. \]
Combining these inequalities yields \eqref{3.9}.
From \eqref{3.9}, we obtain 
\begin{align*}
\frac{(a+b)^p - a^p - b^p}{a^{p-1}b + ab^{p-1}} > \frac{p}{2} \quad \text{ for all } a,b>0.
\end{align*}
By symmetry, we may assume $a \le b$. 
Set $t = \frac{a}{b} \in (0, 1]$. 
Then the above inequality reduces to
\begin{equation*}
	\frac{(t+1)^p - t^p - 1}{t^{p-1}+t} > \frac{p}{2}. 
\end{equation*}
Since $\lim_{t\to 0^+} \frac{(t+1)^p - t^p - 1}{t^{p-1}+t} = p > \frac{p}{2}$, there exists a constant $C_p > \frac{p}{2}$ such that
\[ \inf_{t\in(0,1]} \frac{(t+1)^p - t^p - 1}{t^{p-1}+t} = C_p > \frac{p}{2}, \]
which implies \eqref{3.8}.
\end{proof}

We now prove Proposition \ref{prop 3.2}.

\begin{proof}[Proof of Proposition \ref{prop 3.2}]
Suppose that (a1)--(a3) hold. 
Set $W_\varepsilon \coloneqq U_\varepsilon + U_{\frac{1}{\varepsilon}}$.
By \eqref{3.3}, we have $W_\varepsilon \in \mathcal{D}_K$.
For the estimate of $c_K$, we use $W_\varepsilon$  as a test function. 
By \eqref{3.1} and \eqref{3.2}, we have
\begin{align*}
\int_{\mathbb{R}^N} |\nabla W_\varepsilon|^2 \, dx 
& = \int_{\mathbb{R}^N} |\nabla U_\varepsilon|^2 \, dx + \int_{\mathbb{R}^N} |\nabla U_\frac{1}{\varepsilon}|^2 \, dx + 2 \int_{\mathbb{R}^N} \nabla U_\varepsilon \cdot \nabla U_{\frac{1}{\varepsilon}} \, dx\\
&= 2S^{\frac{N}{2}} + 2 \int_{\mathbb{R}^N} U_\varepsilon^{2^*-1}U_{\frac{1}{\varepsilon}} \, dx.
\end{align*}
As $\varepsilon \to 0^+$, by Lemma \ref{lem 3.3} (ii) with $a(x) \equiv 1$, we obtain
\begin{align*}
\int_{\mathbb{R}^N} |\nabla W_\varepsilon|^2 \, dx &= 2S^{\frac{N}{2}} + 2A \varepsilon^{N-2} + o(\varepsilon^{N-2}).
\end{align*}
By using the expansion $(1+x)^{\frac{N}{2}} = 1 + \frac{N}{2}x + o(|x|)$ as $|x| \to 0$, 
it follows that
\begin{align}
\label{3.10}
\left(\frac{\|\nabla W_\varepsilon\|^2_2}{2S^{\frac{N}{2}}}\right)^{\frac{N}{2}} = \left(1 + \frac{A}{S^{\frac{N}{2}}}\varepsilon^{N-2} + o(\varepsilon^{N-2})\right)^{\frac{N}{2}} = 1 + \frac{NA}{2S^{\frac{N}{2}}}\varepsilon^{N-2} + o(\varepsilon^{N-2}).
\end{align}
Next, by Lemma \ref{lem 3.4}, there exists a constant $C_{2^*} > \frac{2^{*}}{2}$ such that 
\begin{equation}\label{3.11}
\int_{\mathbb{R}^N} a(x) {W_\varepsilon}^{2^*} \, dx \ge \int_{\mathbb{R}^N} a(x)\left(U_\varepsilon^{2^*} + U_{\frac{1}{\varepsilon}}^{2^*} + C_{2^*}U_\varepsilon^{2^*-1}U_{\frac{1}{\varepsilon}} + C_{2^*}U_\varepsilon U_{\frac{1}{\varepsilon}}^{2^*-1}\right) \, dx.
\end{equation}
By Lemma \ref{lem 3.1} (iii) and \eqref{3.3}, we have 
\begin{align*}
\int_{\mathbb{R}^N} a(x) U_\varepsilon^{2^*} \, dx = \int_{\mathbb{R}^N} a(x) (KU_\varepsilon)^{2^*} \, dx = \int_{\mathbb{R}^N} a(x) U_{\frac{1}{\varepsilon}}^{2^*} \, dx.
\end{align*}
By the same change of variable as in the proof of Lemma \ref{lem 3.1} (iii) and \eqref{3.3}, we obtain
\[ \int_{\mathbb{R}^N} a(x) U_\varepsilon^{2^*-1}U_{\frac{1}{\varepsilon}} \, dx = \int_{\mathbb{R}^N} a(x)(KU_\varepsilon)^{2^*-1}(KU_{\frac{1}{\varepsilon}}) \, dx = \int_{\mathbb{R}^N} a(x) U^{2^*-1}_{\frac{1}{\varepsilon}}U_\varepsilon \, dx. \]
Hence \eqref{3.11} becomes
\begin{equation*}
\int_{\mathbb{R}^N} a(x) {W_\varepsilon}^{2^*} \, dx \ge 2 \int_{\mathbb{R}^N} a(x) U_\varepsilon^{2^*} \, dx + 2C_{2^*} \int_{\mathbb{R}^N} a(x) U_\varepsilon^{2^*-1}U_{\frac{1}{\varepsilon}} \, dx.
\end{equation*}
As $\varepsilon \to 0^+$, by Lemma \ref{lem 3.3} (i)--(ii), we obtain
\begin{equation*}
\int_{\mathbb{R}^N} a(x) {W_\varepsilon}^{2^*} \, dx 
\ge 2S^{\frac{N}{2}} + 2C_{2^*}A\varepsilon^{N-2} + o(\varepsilon^{N-2}).
\end{equation*}
By using the expansion $(1+x)^{-\frac{N-2}{2}} = 1 - \frac{N-2}{2}x + o(|x|)$ as $|x| \to 0$, it follows that
\begin{equation}
\label{3.12}
\begin{split}
\left(\frac{\int_{\mathbb{R}^N} a(x){W_\varepsilon}^{2^*} \, dx}{2S^{\frac{N}{2}}}\right)^{-\frac{N-2}{2}} &\le \left(1 + \frac{C_{2^*} A}{S^{\frac{N}{2}}}\varepsilon^{N-2} + o(\varepsilon^{N-2})\right)^{-\frac{N-2}{2}}\\
&= 1 - \frac{(N-2)C_{2^*} A}{2S^{\frac{N}{2}}}\varepsilon^{N-2} + o(\varepsilon^{N-2}).
\end{split}
\end{equation}
Combining \eqref{3.10} and \eqref{3.12}, we obtain
\begin{align*}
c_K &\le \frac{1}{N}\frac{\left(\int_{\mathbb{R}^N} |\nabla W_\varepsilon|^2 \, dx\right)^{\frac{N}{2}}}{\left(\int_{\mathbb{R}^N} a(x) |W_\varepsilon|^{2^*} \, dx\right)^{\frac{N-2}{2}}}\\
&\le \frac{2}{N}S^{\frac{N}{2}}\left(1 + \frac{NA}{2S^{\frac{N}{2}}}\varepsilon^{N-2} + o(\varepsilon^{N-2})\right)
\left(1 - \frac{(N-2)C_{2^*} A}{2S^{\frac{N}{2}}}\varepsilon^{N-2} + o(\varepsilon^{N-2})\right)\\
&= \frac{2}{N}S^{\frac{N}{2}} - (N-2)A\left(C_{2^*} - \frac{N}{N-2}\right)\varepsilon^{N-2} + o(\varepsilon^{N-2}).
\end{align*}
Recalling that $C_{2^*} >\frac{2^*}{2} = \frac{N}{N-2}$, we conclude $c_K < \frac{2}{N}S^{\frac{N}{2}}$.

\medskip

Next, suppose that (a1)--(a2) and (a4) hold. 
In this case, we use $U_1$ as a test function. 
From \eqref{3.3} with $\varepsilon = 1$, we see that $U_1 \in \mathcal{D}_K$.
By the change of variables $x = -\frac{y}{|y|^2}$ and (a2), it follows that
\begin{equation*}
\int_{|x| \ge 1} \frac{a(x)-2^{-\frac{2}{N-2}}}{(1+|x|^2)^N} \, dx = \int_{|y| \le 1} \frac{a(y)-2^{-\frac{2}{N-2}}}{(1+|y|^2)^N} \, dy.
\end{equation*}
Combining this and (a4), we obtain
\begin{equation*}
\int_{\mathbb{R}^N} \frac{a(x)-2^{-\frac{2}{N-2}}}{(1+|x|^2)^N} dx >0.
\end{equation*}
Hence, we deduce
\begin{equation*}
\int_{\mathbb{R}^N} a(x) |U_1|^{2^*} \, dx > 2^{-\frac{2}{N-2}} \int_{\mathbb{R}^N} |U_1|^{2^*} \, dx.
\end{equation*}
Therefore, we have
\begin{align*}
c_K &\le \frac{1}{N} \left(\frac{\int_{\mathbb{R}^N} |\nabla U_1|^2 \, dx}{\left(\int_{\mathbb{R}^N} a(x) |U_1|^{2^*} \, dx\right)^{\frac{2}{2^*}}}\right)^{\frac{N}{2}} \\
&< \frac{1}{N}\cdot\frac{1}{(2^{-\frac{2}{N-2}})^{\frac{2}{2^*}\cdot\frac{N}{2}}} \left(\frac{\int_{\mathbb{R}^N} |\nabla U_1|^2 \, dx}{\left(\int_{\mathbb{R}^N} |U_1|^{2^*} \, dx\right)^{\frac{2}{2^*}}}\right)^{\frac{N}{2}} \\
&= \frac{2}{N}S^{\frac{N}{2}},
\end{align*}
which completes the proof.
\end{proof}

\section{Proof of Theorem \ref{thm 1.2}}
\label{section 4}

In this section, we prove Theorem \ref{thm 1.2}. 
A crucial step in the proof is to establish the compactness of $(PS)_{c_K}$ sequences for $I$.
 Recall that, for $(u_k) \subset \mathcal{D}^{1,2}(\mathbb{R}^N)$ and $d \in \mathbb{R}$, 
we say that $(u_k)$ is a (PS)$_d$ sequence for $I$ if $I(u_k) \to d$ and $I^\prime(u_k) \to 0$ as $k \to \infty$. 
To this end, we classify Kelvin-invariant (PS)$_d$ sequences for $I$ via the profile decomposition given in Theorem \ref{thm 1.8} 
with $H=\mathcal{D}^{1,2}(\mathbb{R}^N)$, $D=G$, and $K$ equal to the Kelvin transform on $\mathcal{D}^{1,2}(\mathbb{R}^N)$. 
To formulate this decomposition, we introduce the auxiliary functional $I_a$. 
For $a \in (0,1]$, we define $I_a$ by
\begin{equation*}
	I_a(u) = \frac{1}{2} \int_{\mathbb{R}^N} |\nabla u|^2 \, dx - \frac{a}{2^*} \int_{\mathbb{R}^N} |u|^{2^*} \, dx 
	\quad (u \in \mathcal{D}^{1,2}(\R^N)). 
\end{equation*}
We also denote by $\mathcal{K}_K$ the set of critical points of $I$ on $\mathcal{D}_K$, and by $\mathcal{K}_a$ the set of critical points of $I_a$:
\begin{align*}
\mathcal{K}_K &= \{u \in \mathcal{D}_K \mid I^\prime(u) = 0\},\\
\mathcal{K}_a &= \{u \in \mathcal{D}^{1,2}(\mathbb{R}^N) \mid I^\prime_a(u) = 0\}.
\end{align*}
Note that, for $u \in \mathcal{K}_a \setminus \{0\}$, we have
\begin{align}
\label{4.1}
I_a(u) =\frac{1}{N}\|u\|_{\mathcal{D}^{1,2}}^2 =  \frac{1}{a^{\frac{N-2}{2}}N}\left(\frac{\|u\|_{\mathcal{D}^{1,2}}}{\|u\|_{2^*}}\right)^N
\ge \frac{1}{N}S^{\frac{N}{2}}.
\end{align}
We state the profile decomposition in our setting.

\begin{prop}
\label{prop 4.1}
Suppose (a1)--(a2) and let $({u}_k) \subset \mathcal{D}_K$ be a (PS)$_{d}$ sequence for $I$. Then, up to a subsequence, there exist $n_0 \in \mathbb{N} \cup \{0\}$, ${w}^{(0)} \in \mathcal{K}_K$ and ${g}_k^{(0)} = id$, and for each $n\in \{1,\dots,n_0\}$, there exist $a^{(n)} \in (0,1]$, $w^{(n)} \in \mathcal{K}_{a^{(n)}} \setminus \{0\}$, and $(g_k^{(n)}) \subset G$ such that
\begin{align}
\label{4.2}
& {{g}_k^{(n)}}^* {u}_k \rightharpoonup {w}^{(n)} \text{ weakly in } \mathcal{D}^{1,2}(\mathbb{R}^N) \text{ for all } n\in \{0,1,\dots,n_0\},\\
\label{4.3}
& {{g}_k^{(n)}}^* {g}_k^{(m)} \rightharpoonup 0 \text{ weakly for all } n,m\in \{0,1,\dots,n_0\} \text{ with } n\neq m,\\
\label{4.4}
& {{g}_k^{(n)}}^* K {g}_k^{(m)} \rightharpoonup 0 \text{ weakly for all } n,m\in \{1,\dots,n_0\},\\
\label{4.5}
& \left\| {u}_k - {w}^{(0)} - \sum_{n=1}^{n_0} ({g}_k^{(n)}{w}^{(n)} + {K}{g}_k^{(n)}{w}^{(n)}) \right\|_{\mathcal{D}^{1,2}} \to 0,\\
\label{4.6}
& d = \lim_{k\to\infty} I(u_k) = I(w^{(0)}) + 2\sum_{n=1}^{n_0} I_{a^{(n)}}(w^{(n)}).
\end{align}
\end{prop}

In the proof of Proposition \ref{prop 4.1}, we use the following two lemmas. 
The first lemma shows that $G$-weak convergence implies strong convergence in $L^{2^*}(\R^N)$. 
Although the proof follows from arguments in \cite{Tintarev07}, we include it in Appendix \ref{Appendix A} for completeness.

\begin{lem}
\label{lem 4.2}
Let $({u}_k) \subset \mathcal{D}^{1,2}(\mathbb{R}^N)$. 
If ${u}_k \overset{G}{\rightharpoonup} 0$, then ${u}_k \to 0 \text{ in } L^{2^*}(\mathbb{R}^N)$. 
\end{lem}

The second lemma is a variant of the Brezis--Lieb lemma \cite{Brezis--Lieb}. 
Here, set $KG=\left\{Kg\mid g\in G\right\}$.

\begin{lem}
\label{lem 4.3}
Let $a \in L^\infty(\mathbb{R}^N)$, $(g_k)\subset G\cup KG$, and $(v_k)\subset \mathcal{D}^{1,2}(\R^N)$.
Suppose that $g_k^*v_k \rightharpoonup v_0$ weakly in $\mathcal{D}^{1,2}(\mathbb{R}^N)$ for some $v_0\in \mathcal{D}^{1,2}(\R^N)$.
Then we have
\begin{equation}\label{4.7}
	\int_{\mathbb{R}^N} a(x) |v_k - g_kv_0|^{2^*} \, dx 
	= \int_{\mathbb{R}^N} a(x) |v_k|^{2^*} \, dx - \int_{\mathbb{R}^N} a(x) |g_kv_0|^{2^*} \, dx + o(1)
\end{equation}
as $k \to \infty$.
\end{lem}

\begin{proof}
Since $(g_k)\subset G\cup KG$ consists of isometries on $L^{2^*}(\mathbb{R}^N)$, we have
\begin{align*}
	\int_{\R^N}\left||v_k|^{2^*} - |v_k - g_kv_0|^{2^*} - |g_kv_0|^{2^*}\right|\, dx 
	= \int_{\R^N}\left||g_k^* v_k|^{2^*} - |g_k^*v_k - v_0|^{2^*} - |v_0|^{2^*}\right|\, dx.
\end{align*}
Passing to a subsequence, we may assume that $g_k^*v_k \to v_0$ a.e. $x\in \R^N$. 
By the Brezis-Lieb lemma \cite{Brezis--Lieb}, we have
\begin{align*}
	\int_{\R^N}\left||g_k^* v_k|^{2^*} - |g_k^*v_k - v_0|^{2^*} - |v_0|^{2^*}\right|\, dx = o(1)
\end{align*}
as $k\to \infty$.
Therefore we obtain
\begin{align*}
	& \left|\int_{\mathbb{R}^N} a(x) |v_k|^{2^*} \, dx - \int_{\mathbb{R}^N} a(x) |v_k - g_kv_0|^{2^*} \, dx 
	- \int_{\mathbb{R}^N} a(x) |g_kv_0|^{2^*} \, dx\right| \\
	& \le \|a\|_\infty \int_{\R^N}\left||v_k|^{2^*} - |v_k - g_kv_0|^{2^*} - |g_kv_0|^{2^*}\right|\, dx = o(1).
\end{align*}
Since the above argument applies to every subsequence, the original sequence also converges to $0$.
This implies \eqref{4.7}.
\end{proof}

\begin{proof}[Proof of Proposition \ref{prop 4.1}]We recall that the assumptions of Theorem \ref{thm 1.8} are satisfied with $H=\mathcal{D}^{1,2}(\mathbb{R}^N)$, $D=G$, and $K$ equal to the Kelvin transform on $\mathcal{D}^{1,2}(\mathbb{R}^N)$. 
Let $(u_k) \subset \mathcal{D}_K$ be a (PS)$_d$ sequence for $I$. 
Then
\begin{align*}
\|u_k\|_{\mathcal{D}^{1,2}}^2 &= N\left(I(u_k) - \frac{1}{2^*} I^\prime(u_k)u_k\right)\\
&\le Nd + \frac{N}{2^*} \|I^\prime(u_k)\|_{(\mathcal{D}^{1,2}(\R^N))^*} \|u_k\|_{\mathcal{D}^{1,2}} + o(1).
\end{align*}
Hence  $(u_k)$ is bounded in $\mathcal{D}^{1,2}(\mathbb{R}^N)$. 
By Theorem \ref{thm 1.8}, up to a subsequence, there exist $n_0 \in \mathbb{N} \cup \{0, \infty\}$, ${w}^{(0)} \in \mathcal{D}_K$, 
${g}_k^{(0)} = id$, and for $n\in \{1,\dots,n_0\}$, there exist 
${w}^{(n)} \in \mathcal{D}^{1,2}(\mathbb{R}^N) \setminus \{0\}$ and $({g}_k^{(n)}) \subset G$ such that
\begin{align}
\label{4.8}
& {{g}_k^{(n)}}^* {u}_k \rightharpoonup {w}^{(n)} \text{ weakly in } \mathcal{D}^{1,2}(\mathbb{R}^N) \text{ for all } n\in \{0,1,\dots,n_0\},\\
\label{4.9}
& {{g}_k^{(n)}}^* {g}_k^{(m)} \rightharpoonup 0 \text{ weakly for all } n,m\in \{0,1,\dots,n_0\} \text{ with } n\neq m,\\
\label{4.10}
& {{g}_k^{(n)}}^* K {g}_k^{(m)} \rightharpoonup 0 \text{ weakly for all } n,m\in \{1,\dots,n_0\},\\
\label{4.11}
& \|{w}^{(0)}\|_{\mathcal{D}^{1,2}}^2 + 2 \sum_{n=1}^{n_0} \|{w}^{(n)}\|_{\mathcal{D}^{1,2}}^2 \le \limsup_{k\to\infty} \|{u}_k\|_{\mathcal{D}^{1,2}}^2.
\end{align}
We divide the proof into several steps.

\medskip

\noindent
\textbf{Step 1.} \textit{We have $n_0 < \infty$. Moreover, $w^{(0)}\in \mathcal{K}_K$, 
and for each $n\in \{1,\dots,n_0\}$, there exists $a^{(n)} \in (0, 1]$ such that $w^{(n)}\in \mathcal{K}_{a^{(n)}}\setminus\{0\}$. }
\begin{proof}[Proof of Step 1.]
From the definition of $G$, for each $n\in \{1,\dots,n_0\}$, there exist $(y_k^{(n)}) \subset \mathbb{R}^N$ and $(\lambda_k^{(n)}) \subset \R $ such that $g_k^{(n)} = T_{y_k^{(n)},\lambda_k^{(n)}}$. 
Let $\varphi \in \mathcal{D}^{1,2}(\mathbb{R}^N)$.
Since $(u_k)$ is a (PS)$_d$ sequence, for each $n\in \{0,1,\dots,n_0\}$, we have
\begin{align*}
o(1) &= I^\prime(u_k)(g_k^{(n)}\varphi)\\
&= \int_{\mathbb{R}^N} \nabla u_k \cdot \nabla (g_k^{(n)}\varphi) \, dx - \int_{\mathbb{R}^N} a(x) |u_k|^{2^*-2}u_k(g_k^{(n)}\varphi) \, dx\\
&= \int_{\mathbb{R}^N} \nabla {g_k^{(n)}}^* u_k \cdot \nabla \varphi \, dx - \int_{\mathbb{R}^N} a(2^{-\lambda_k^{(n)}} x+y_k^{(n)}) |{g_k^{(n)}}^* u_k|^{2^*-2}({g_k^{(n)}}^* u_k)\varphi \, dx.
\end{align*}
If $n = 0$, then $g_k^{(0)}=id$ and $g_k^{(0)*}u_k=u_k \rightharpoonup {w}^{(0)} $ weakly in $\mathcal{D}^{1,2}(\mathbb{R}^N)$. 
Thus, we have
\begin{align*}
\int_{\mathbb{R}^N} \nabla w^{(0)} \cdot \nabla \varphi \, dx - \int_{\mathbb{R}^N} a(x) |w^{(0)}|^{2^*-2}w^{(0)}\varphi \, dx = 0.
\end{align*}
In particular, $w^{(0)}\in \mathcal{K}_K$. 
Next, let $n\in \{1,\dots,n_0\}$. 
Since $g_k^{(n)} \rightharpoonup 0$ weakly by \eqref{4.9}, 
Lemma \ref{lem 2.2} yields $|y_k^{(n)}| + |\lambda_k^{(n)}| \to \infty$. 
Passing to a subsequence, we may assume that either $|y_k^{(n)}| \to \infty$, or $|\lambda_k^{(n)}|\to \infty$ and  $y_k^{(n)} \to y^{(n)}$.
If $\lambda_k^{(n)} \to -\infty$ or $|y_k^{(n)}| \to \infty$, then $a(2^{-\lambda_k^{(n)}}x+y_k^{(n)}) \to 1$ a.e. $x\in \R^N$.
By \eqref{4.8}, we obtain
\begin{align*}
\int_{\mathbb{R}^N} \nabla w^{(n)} \cdot \nabla \varphi \, dx - \int_{\mathbb{R}^N} |w^{(n)}|^{2^*-2}w^{(n)}\varphi \, dx = 0.
\end{align*}
Hence $a^{(n)} := 1$ and $w^{(n)}\in \mathcal{K}_{a^{(n)}}\setminus\{0\}$. 
On the other hand, if $\lambda_k^{(n)} \to \infty$ and $y_k^{(n)} \to y^{(n)}$, then $a(2^{-\lambda_k^{(n)}}x+y_k^{(n)}) \to a(y^{(n)})$  a.e. $x \in \mathbb{R}^N$. By \eqref{4.8}, we have
\begin{align*}
\int_{\mathbb{R}^N} \nabla w^{(n)} \cdot \nabla \varphi \, dx - \int_{\mathbb{R}^N} a(y^{(n)}) |w^{(n)}|^{2^*-2}w^{(n)}\varphi \, dx = 0.
\end{align*}
Hence $a^{(n)} := a(y^{(n)})\in (0,1]$ and $w^{(n)}\in \mathcal{K}_{a^{(n)}}\setminus\{0\}$.  
By \eqref{4.1}, we have 
\begin{equation*}
\|w^{(n)}\|_{\mathcal{D}^{1,2}}^2 \ge S^{\frac{N}{2}} \quad \text{for all } n \in \{1,\dots,n_0\} 
\end{equation*}
Combining this with \eqref{4.11}, we conclude that $n_0 < \infty$. 
\end{proof}

\noindent
\textbf{Step 2.} \textit{\eqref{4.5} holds.}
\begin{proof}[Proof of Step 2.]
Set
\begin{align*}
r_k \coloneqq {u}_k - {w}^{(0)} - \sum_{n=1}^{n_0} ({g}_k^{(n)}{w}^{(n)} + {K}{g}_k^{(n)}{w}^{(n)}).
\end{align*}
Since $n_0 < \infty$, by \eqref{1.8} in Theorem \ref{thm 1.8}, we have ${r}_k \overset{G}{\rightharpoonup} 0$. Thus, Lemma \ref{lem 4.2} implies 
\begin{align}
\label{4.12}
r_k \to 0 \quad \text{in } L^{2^*}(\mathbb{R}^N).
\end{align}
Since ${u}_k$ is a (PS)$_d$ sequence, we have $I^\prime({u}_k)r_k \to 0$. Furthermore, $w^{(0)} \in \mathcal{K}_K$ implies $I^\prime({w}^{(0)})({r}_k) = 0$. Moreover, for each $n\in \{1,\dots,n_0\}$, since $w^{(n)} \in \mathcal{K}_{a^{(n)}}$ and $I_{a^{(n)}}$ is invariant under conformal transformations, we have
\begin{align*}
& I^\prime_{a^{(n)}}(g_k^{(n)}w^{(n)})r_k = I^\prime_{a^{(n)}}({w}^{(n)})({{g}_k^{(n)}}^*{r}_k) = 0,\\
& I^\prime_{a^{(n)}}(Kg_k^{(n)}w^{(n)})r_k = I^\prime_{a^{(n)}}({w}^{(n)})({{g}_k^{(n)}}^* {K}{r}_k) = 0.
\end{align*}
Therefore, we obtain 
\begin{align*}
o(1) &= I^\prime(u_k)r_k - I^\prime(w^{(0)})r_k - \sum_{n=1}^{n_0} \left\{I^\prime_{a^{(n)}}(g_k^{(n)} w^{(n)})r_k + I^\prime_{a^{(n)}}(K g_k^{(n)} w^{(n)})r_k\right\}\\
&= \|r_k\|^2_{\mathcal{D}^{1,2}} - \int_{\mathbb{R}^N} a(x) |u_k|^{2^*-2}u_kr_k \, dx + \int_{\mathbb{R}^N} a(x) |w^{(0)}|^{2^*-2}w^{(0)}r_k \, dx\\
&\qquad\qquad + \sum_{n=1}^{n_0} \int_{\mathbb{R}^N} a^{(n)} |g_k^{(n)}w^{(n)}|^{2^*-2}g_k^{(n)}w^{(n)}r_k \, dx\\
&\qquad\qquad + \sum_{n=1}^{n_0} \int_{\mathbb{R}^N} a^{(n)} |Kg_k^{(n)}w^{(n)}|^{2^*-2}Kg_k^{(n)}w^{(n)}r_k \, dx
\end{align*}
By \eqref{4.12}, all integral terms are $o(1)$. Hence, $\|r_k\|^2_{\mathcal{D}^{1,2}} = o(1)$. Therefore, \eqref{4.5} holds. 
\end{proof}

\noindent
\textbf{Step 3.} \textit{\eqref{4.6} holds. }
\begin{proof}[Proof of Step 3.]
First, using \eqref{4.8}--\eqref{4.10}, we deduce from \eqref{4.5} that
\begin{align}
\label{4.13}
\lim_{k\to\infty} \|u_k\|_{\mathcal{D}^{1,2}}^2 = \|{w}^{(0)}\|_{\mathcal{D}^{1,2}}^2 + 2 \sum_{n=1}^{n_0} \|{w}^{(n)}\|_{\mathcal{D}^{1,2}}^2.
\end{align}
Next, we prove that
\begin{align}
\label{4.14}
	\lim_{k\to\infty} \int_{\mathbb{R}^N} a(x) |u_k|^{2^*} \, dx 
	= \int_{\mathbb{R}^N} a(x) |w^{(0)}|^{2^*} \, dx + 2\sum_{n=1}^{n_0} \int_{\mathbb{R}^N} a^{(n)} |w^{(n)}|^{2^*} \, dx
\end{align}
By \eqref{4.12}, we have
\begin{equation}\label{4.15}
	\int_{\mathbb{R}^N} a(x) |r_k|^{2^*}\, dx
%	=\int_{\mathbb{R}^N} a(x) \left| u_k - w^{(0)} - \sum_{n=1}^{n_0} (g_k^{(n)}w^{(n)} + Kg_k^{(n)}w^{(n)})\right|^{2^*} \, dx 
	= o(1).
\end{equation}
Using \eqref{4.8}--\eqref{4.10},
Lemma \ref{lem 4.3} can be applied successively
to remove each profile from the decomposition of $u_k$.
Consequently,
\begin{align}
	\int_{\R^N} a(x) |r_k|^{2^*}\, dx
	= & \int_{\R^N} a(x) \left| u_k \right|^{2^*} \, dx 
	- \int_{\R^N} a(x) |w^{(0)}|^{2^*} \, dx \notag \\
	& - \sum_{n=1}^{n_0} \int_{\mathbb{R}^N} a(x) |g_k^{(n)}w^{(n)}|^{2^*} \, dx 
	- \sum_{n=1}^{n_0} \int_{\mathbb{R}^N} a(x) |Kg_k^{(n)}w^{(n)}|^{2^*} \, dx + o(1) \label{4.16}.
\end{align}
From a change of variables and the definition of $a^{(n)}$, it follows  
\begin{align}
	& \int_{\R^N} a(x) |g_k^{(n)}w^{(n)}|^{2^*} \, dx 
	 =\int_{\R^N} a(2^{-\lambda_k^{(n)}} x + y_k^{(n)}) |w^{(n)}|^{2^*} \, dx 
	= a^{(n)} \|w^{(n)}\|_{2^*}^{2^*} + o(1), \label{4.17} \\
	& \int_{\R^N} a(x) |Kg_k^{(n)}w^{(n)}|^{2^*} \, dx 
	=\int_{\R^N} a(x) |g_k^{(n)}w^{(n)}|^{2^*} \, dx 
	= a^{(n)} \|w^{(n)}\|_{2^*}^{2^*} + o(1). \label{4.18}
\end{align}
Combining \eqref{4.15}, \eqref{4.16}, \eqref{4.17}, and \eqref{4.18}, we obtain \eqref{4.14}.
From \eqref{4.13} and \eqref{4.14}, it follows \eqref{4.6}.
\end{proof}
Steps 1--3 complete the proof.
\end{proof}

\begin{proof}[Proof of Theorem \ref{thm 1.2}]
Since $c_K$ is characterized as the mountain pass value of $I$,
there exists a (PS)$_{c_K}$ sequence $(u_k)$ for $I$. 
Passing to a subsequence as in Proposition \ref{prop 4.1}, by \eqref{4.1} and \eqref{4.6}, we obtain
\begin{equation*}
	c_K = \lim_{k\to\infty} I(u_k) = I(w^{(0)}) + 2\sum_{n=1}^{n_0} I_{a^{(n)}}(w^{(n)}) \ge \frac{2 n_0}{N}S^{\frac{N}{2}}.
\end{equation*}
On the other hand, Proposition \ref{prop 3.2} yields $c_K < \frac{2}{N}S^{\frac{N}{2}}$. 
Therefore, we must have $n_0 = 0$, and by \eqref{4.5}, we obtain ${u}_k \to {w}^{(0)}$ in $ \mathcal{D}^{1,2}(\mathbb{R}^N)$. 
Hence, $w^{(0)} \in \mathcal{D}_K$ is a minimizer of $c_K$ and a Kelvin-invariant solution of \eqref{Pa}.
Replacing $w^{(0)}$ by $|w^{(0)}|$ if necessary, we may assume that $w^{(0)}$ is nonnegative.
By the strong maximum principle, $w^{(0)}$ is positive.
Therefore, the proof is complete.
\end{proof}

\subsection*{Acknowledgements}
This work was supported by JSPS KAKENHI Grant Numbers JP20K03691.

\begin{appendix}

\section{Proof of Lemma \ref{lem 4.2}}
\label{Appendix A}

In this appendix, we provide the proof of Lemma \ref{lem 4.2}. Let $\chi \in C_0^\infty\left(\mathbb{R},[0,\infty)\right)$ such that 
\begin{align*}
	& \chi(t) = |t| \quad \text{ if } |t|\in [1,2],\\
	& \chi(t) = 0 \quad \text{ if } |t| \not\in \left[\frac{1}{2},4\right].
\end{align*}
For $k \in \mathbb{Z}$, set
\begin{equation*}
 \chi_k(t) = 2^k \chi(2^{-k}t).
\end{equation*}
Then, there exists $C_1$, $C_2 > 0$ such that for all $t \in \mathbb{R}$,
\begin{align}
\label{A.1}
\chi(t)^{2^*} &\le C_1 \chi(t)^2,\\
\label{A.2}
\chi(t)^{2} &\le C_2 \left(\chi_{-1}(t)^{2^*} + \chi_{0}(t)^{2^*} + \chi_{1}(t)^{2^*} \right).
\end{align}
Furthermore, let $Q = [0,1]^N$, and for each $y = (y_1, \dots, y_N) \in \mathbb{Z}^N$, define
\begin{equation*}
Q + y = [y_1, y_1+1] \times \cdots \times [y_N, y_N+1].
\end{equation*}
For $u:\mathbb{R}^N \to \mathbb{R}$ and $a < b$, we use the notation 
\begin{equation*}
[a \le u \le b] = \{x \in \mathbb{R}^N \mid a \le u(x) \le b\}.
\end{equation*}
Moreover, for $\lambda \in \R$ and $y \in \mathbb{R}^N$, we define a rescaled dilation operator $\widetilde{\delta}_\lambda$ by 
\begin{equation*}
	(\widetilde{\delta}_\lambda u)(x) = (\delta_{-\frac{2^*}{N}\lambda} u)(x) = 2^{-\lambda} u(2^{-\frac{2^*}{N}\lambda} x),
\end{equation*}
and define $\widetilde{T}_{y,\lambda} = \tau_y \circ \widetilde{\delta}_{\lambda}$.
Then, $G$ is written as
\begin{equation*}
G = \{ \widetilde{T}_{y,\lambda} \mid y \in \mathbb{R}^N, \lambda \in \mathbb{R}\}.
\end{equation*}
In this section, for $A\subset \R^N$ and $p\ge 1$, we use the following notation.
\begin{align*}
\|u\|_{L^p(A)}^p  = \int_A |u|^p \, dx,\ \qquad
\|u\|_{H^1(A)}^2 = \int_{A}|\nabla u|^2 + u^2\, dx.
\end{align*}

\begin{lem}
\label{lem A.1}
There exists $C_3>0$ such that, for all $u \in \mathcal{D}^{1,2}(\mathbb{R}^N)$,
\begin{equation}\label{A.3}
	\|\chi(u)\|_{H^1(\R^N)}^2 \le C_3 \left(\|\nabla u\|_{L^2([2^{-1}\le |u| \le 2^2])}^2 
	+ \|\nabla u\|_{L^2([2^{-2}\le |u| \le 2^3])}^{2^*} \right)
\end{equation}
\end{lem}

\begin{proof}
A direct calculation shows
\begin{align*}
	\|\chi(u)\|_{H^1(\R^N)}^2 
	& = \|\nabla (\chi(u))\|_{L^2(\R^N)}^2 + \|\chi(u)\|_{L^2(\R^N)}^2 \\
	& = \|\chi'(u) \nabla u\|_{L^2(\R^N)}^2 + \|\chi(u)\|_{L^2(\R^N)}^2.
\end{align*}
Using \eqref{A.2} and the Sobolev inequalities, we obtain
\begin{align*}
	\|\chi(u)\|_{L^2(\R^N)}^2 & \le C_2 \sum_{k=-1,0,1}\| \chi_k(u)\|_{L^{2^*}(\R^N)}^{2^*} \\
	& \le C_2 S^{-\frac{2^*}{2}} \sum_{k=-1,0,1}\| \nabla(\chi_k(u))\|_{L^2(\R^N)}^{2^*} \\
	& =  C_2 S^{-\frac{2^*}{2}} \sum_{k=-1,0,1} \|\chi_k'(u) \nabla u\|_{L^2(\R^N)}^{2^*}. 
\end{align*}
Since $\chi'_k$ is bounded and
${\rm supp}(\chi_k'(u))
\subset [2^{k-1}\le |u| \le 2^{k+2}]$ for all $k \in \Z$,
the desired estimate \eqref{A.3} follows.
\end{proof}

We use the following lemma in the proof of Lemma \ref{lem 4.2}.

\begin{lem}
\label{lem A.2}
For any $M>0$, there exists $C>0$ such that, for all $u \in \mathcal{D}^{1,2}(\mathbb{R}^N)$ with $\|\nabla u\|_{L^2(\R^N)}\le M$,
\begin{equation}
\label{A.4}
\|u\|_{L^{2^*}(\R^N)}^{2^*} \le C \left(\sup_{j\in \Z}\sup_{y\in\mathbb{Z}^N} \|\chi(\widetilde{T}_{y,j}u)\|_{L^2(Q)}\right)^{2-\frac{4}{2^*}}.
\end{equation}
\end{lem}

\begin{proof}
We divide the proof into several steps.

\medskip

\noindent
\textbf{Step 1.} \textit{There exists $C' > 0$ such that, for all $u \in \mathcal{D}^{1,2}(\mathbb{R}^N)$ with $\|\nabla u\|_{L^2(\R^N)}\le M$,
\begin{equation}
\label{A.5}
\|\chi(u)\|_{L^{2^*}(\R^N)}^{2^*} 
\le C' \left(\sup_{y\in\mathbb{Z}^N} \|\chi(\tau_y u)\|_{L^2(Q)}\right)^{2-\frac{4}{2^*}}
\|\nabla u\|_{L^2([2^{-2} \le |u| \le 2^{3}])}^2.
\end{equation}}
\begin{proof}[Proof of Step 1.]
From the Sobolev inequality, there exists $C_4>0$ such that $\|u\|_{L^{2^*}(Q)}^2\le C_4\|u\|_{H^1(Q)}^2$.
Thus, we have
\begin{align}
\|\chi(u)\|_{L^{2^*}(\R^N)}^{2^*} 
& =\sum_{y\in \Z^N}\|\chi(u)\|_{L^{2^*}(Q+y)}^{2^*} \notag \\
& \le \left(\sup_{y\in\mathbb{Z}^N} \|\chi(u)\|_{L^{2^*}(Q+y)}\right)^{2^*-2}\sum_{y\in \Z^N}\|\chi(u)\|_{L^{2^*}(Q+y)}^2 \notag \\
& \le C_4\left(\sup_{y\in\mathbb{Z}^N} \|\chi(u)\|_{L^{2^*}(Q+y)}\right)^{2^*-2}\sum_{y\in \Z^N}\|\chi(u)\|_{H^1(Q+y)}^2 \notag \\
\label{A.6} & = C_4\left(\sup_{y\in\mathbb{Z}^N} \|\chi(\tau_y u)\|_{L^{2^*}(Q)}\right)^{2^*-2}\|\chi(u)\|_{H^1(\R^N)}^2.
\end{align}
Using \eqref{A.1}, we have
\begin{equation}\label{A.7}
	\|\chi(\tau_y u)\|_{L^{2^*}(Q)}^{2^*} \le C_1\|\chi(\tau_y u)\|_{L^2(Q)}^2.
\end{equation}
Combining \eqref{A.6}, \eqref{A.7}, and \eqref{A.3}, we obtain \eqref{A.5}.
\end{proof}

\noindent
\textbf{Step 2.} \textit{There exists $C''>0$ such that for all $j \in \mathbb{Z}$ and $u \in \mathcal{D}^{1,2}(\mathbb{R}^N)$ with $\|\nabla u\|_{L^2(\R^N)}\le M$, 
\begin{equation}
\label{A.8}
\|\chi_j (u)\|_{L^{2^*}(\R^N)}^{2^*} 
\le C'' \left(\sup_{y\in\mathbb{Z}^N} \|\chi(\widetilde{T}_{y,j} u)\|_{L^2(Q)}\right)^{2-\frac{4}{2^*}}
\|\nabla u\|_{L^2([2^{j-2} \le |u| \le 2^{j+3}])}^2. 
\end{equation}}
\begin{proof}[Proof of Step 2.]
By direct calculation, for $j\in \Z$ and $u \in \mathcal{D}^{1,2}(\mathbb{R}^N)$ with $\|\nabla u\|_{L^2(\R^N)}\le M$, we have
\begin{align}
	\|\nabla (\widetilde{\delta}_j u)\|_{L^2(\R^N)}^2 & = \|\nabla u\|_{L^2(\R^N)}^2,\\
	\|\nabla (\widetilde{\delta}_j u)\|_{L^2([2^{-2} \le |\widetilde{\delta}_j u| \le 2^3])}^2 
	& = \|\nabla u\|_{L^2([2^{j-2} \le |u| \le 2^{j+3}])}^2, \\
	\|\chi (\widetilde{\delta}_j u)\|_{L^{2^*}(\R^N)}^{2^*} & = \|\chi_j(u)\|_{L^{2^*}(\R^N)}^{2^*}.
\end{align}
Applying \eqref{A.5} to $\widetilde{\delta}_j u$, we obtain \eqref{A.8}.
\end{proof}

\noindent
\textbf{Step 3.} \textit{Conclusion.}
\begin{proof}[Proof of Step 3.]
Since $\chi_j(t)=|t|$ for $|t|\in[2^j,2^{j+1}]$, we have
\begin{equation}
\label{A.12}
\|u\|_{L^{2^*}([2^j \le |u| \le 2^{j+1}])}^{2^*} \le \|\chi_j(u)\|_{L^{2^*}(\R^N)}^{2^*}\quad\text{for all } j \in \mathbb{Z}.
\end{equation}
Combining \eqref{A.8} and \eqref{A.12}, we have
\begin{align*}
	\|u\|_{L^{2^*}(\R^N)}^{2^*} 
	& = \sum_{j \in \mathbb{Z}} \|u\|_{L^{2^*}([2^j \le |u| \le 2^{j+1}])}^{2^*}\\
	& \le  \sum_{j \in \mathbb{Z}} \|\chi_j(u)\|_{L^{2^*}(\R^N)}^{2^*}\\
	&\le C'' \sum_{j \in \mathbb{Z}} \left(\sup_{y\in\mathbb{Z}^N} \|\chi(\widetilde{T}_{y,j} u)\|_{L^2(Q)}\right)^{2-\frac{4}{2^*}}
	\|\nabla u\|_{L^2([2^{j-2} \le |u| \le 2^{j+3}])}^2 \\
	&\le C'' \left(\sup_{j\in \Z}\sup_{y\in\mathbb{Z}^N} \|\chi(\widetilde{T}_{y,j} u)\|_{L^2(Q)}\right)^{2-\frac{4}{2^*}}
	\sum_{j\in \Z}\|\nabla u\|_{L^2([2^{j-2} \le |u| \le 2^{j+3}])}^2 \\
	&\le 5C'' \left(\sup_{j\in \Z}\sup_{y\in\mathbb{Z}^N} \|\chi(\widetilde{T}_{y,j} u)\|_{L^2(Q)}\right)^{2-\frac{4}{2^*}}\|\nabla u\|_{L^2(\R^N)}^2.
\end{align*}
Thus, \eqref{A.4} holds with $C =  5C''M^2$.
\end{proof}
Steps 1--3 complete the proof.
\end{proof}

\begin{proof}[Proof of Lemma \ref{lem 4.2}]
Assume that ${u}_k \overset{G}{\rightharpoonup} 0$. Since $u_k \rightharpoonup 0$ weakly in $\mathcal{D}^{1,2}(\mathbb{R}^N)$, we see that $(u_k)$ is bounded in $\mathcal{D}^{1,2}(\mathbb{R}^N)$. 
By Lemma \eqref{lem A.2}, we have
\begin{equation}
\label{A.13}
\|u_k\|_{L^{2^*}(\R^N)}^{2^*} \le C \left(\sup_{j\in \Z}\sup_{y\in\mathbb{Z}^N} \|\chi(\widetilde{T}_{y,j} u_k)\|_{L^2(Q)}\right)^{2-\frac{4}{2^*}} \text{for all $k \in \mathbb{N}$},
\end{equation}
where $C>0$ is the constant appearing in Lemma \ref{lem A.2}.
Then, there exist $(y_k) \subset \Z^N$ and $(j_k) \subset \mathbb{Z}$ such that
\begin{equation}
\label{A.14}
	\sup_{j\in\mathbb{Z}}\sup_{y \in \mathbb{Z}^N} \|\chi(\widetilde{T}_{y,j} u_k)\|_{L^2(Q)} 
	\le 2\|\chi(\widetilde{T}_{y_k,j_k} u_k)\|_{L^2(Q)}.
\end{equation}
Since ${u}_k \overset{G}{\rightharpoonup} 0$, we have 
$\widetilde{T}_{y_k,j_k}u_k \rightharpoonup 0$ weakly in $\mathcal{D}^{1,2}(\mathbb{R}^N)$.
By the local compactness of the Sobolev embedding, it follows that
\begin{equation*}
\widetilde{T}_{y_k,j_k}u_k \to 0 \quad \text{in } L^{2}(Q).
\end{equation*}
Hence, we have
\begin{equation*}
\|\chi(\widetilde{T}_{y_k,j_k} u_k)\|_{L^2(Q)} \to 0.
\end{equation*}
Therefore, by \eqref{A.14}, we obtain
\begin{equation}
\label{A.15}
\sup_{j\in\mathbb{Z}}\sup_{y \in \mathbb{Z}^N} \|\chi(\widetilde{T}_{y,j} u_k)\|_{L^2(Q)} \to 0.
\end{equation}
Combining \eqref{A.13} and \eqref{A.15}, we conclude that $u_k \to 0$ in $L^{2^*}(\mathbb{R}^N)$.
\end{proof}
\end{appendix}

\subsection*{Data Availability}
No data was used for the research described in the article.

\end{document}